\documentclass[12pt]{article}
\usepackage{geometry}
\usepackage{amsmath, amssymb, amsthm}

\usepackage{graphicx}
\usepackage{color}
\usepackage{subfigure}

\usepackage[utf8]{inputenc}
\usepackage{latexsym}
\newtheorem{theorem}{Theorem}[section]

\usepackage{hyperref}
\usepackage{graphicx,epsfig}
\usepackage{color}
\usepackage{epstopdf}
\usepackage{caption}
\usepackage{float}
\usepackage{subfigure}

\begin{document}

{\LARGE \bf  
\begin{center}
Oscillations and pattern formation in\\ a slow-fast prey-predator system
\end{center}
}


\vspace*{1cm}

\centerline{\bf Pranali Roy Chowdhury$^1$, Sergei Petrovskii$^{2,3}$, Malay Banerjee$^1$}

\vspace{0.5cm}

\centerline{ $^1$ Indian Institute of Technology Kanpur, Kanpur - 208016, India}

\centerline{ $^2$ School of Mathematics \& Actuarial Science, University of Leicester, Leicester LE1 7RH, UK}

\centerline{ $^3$ Peoples Friendship University of Russia (RUDN University), 6 Miklukho-Maklaya St,} 
\centerline{ Moscow 117198, Russian Federation}

\vspace{1cm}

\noindent
\begin{abstract}
	We consider the properties of a slow-fast prey-predator system in time and space. We first argue that the simplicity of prey-predator system is apparent rather than real and there are still many of its hidden properties that have been poorly studied or overlooked altogether. We further focus on the case where, in the slow-fast system, the prey growth is affected by a weak Allee effect. We first consider this system in the non-spatial case and make its comprehensive study using a variety of mathematical techniques. In particular, we show that the interplay between the Allee effect and the existence of multiple timescales may lead to a regime shift where small-amplitude oscillations in the population abundances abruptly change to large-amplitude oscillations. 
	We then consider the spatially explicit slow-fast prey-predator system and reveal the effect of different time scales on the pattern formation. We show that a decrease in the timescale ratio may lead to another regime shift where the spatiotemporal pattern becomes spatially correlated leading to large-amplitude oscillations in spatially average population densities and potential species extinction. 
	
	\vspace{0.5cm}
	
	\noindent
	{\bf Keywords:} Slow-fast time scale; relaxation oscillation; canard cycle; spatial pattern; regime shift 
	
\end{abstract}

\section{Introduction}
\label{intro}
In the natural environment, interactions in a population community are usually quite complex \cite{0b,0a,0c}. This ubiquitous complexity have several different sources such as the complexity of the wood web, nonlinearity of species feedbacks, multiplicity of temporal and spatial scales, etc.
It is extremely difficult, in fact hardly possible at all to capture the entire complexity of ecological interactions in a single mathematical model or framework. Instead, the usual means of analysis tend to focus on a particular aspect or feature of the ecological system. For instance, while the food web theory endeavours to link the properties of a realistic population community to the complexity of the corresponding food web, in particular by analysing the web connectivity and revealing the bottlenecks \cite{Allesina04,Polis96}, a lot of attention focuses on the properties of simpler `building blocks' from which the web is made \cite{Jordan02}. A variety of blocks of intermediate complexity have been considered, a few examples are given by the three-species competition system \cite{Hofbauer98}, intraguild predation \cite{Holt97} and a three-species resource-consumer-predator food chain \cite{Hastings91}. 

Arguably, the most basic block is the prey-predator system. 
It has been a focus of research for almost a century \cite{Rosenzweig63,Volterra26} and there is a tendency to think about it as a fully studied, textbook material \cite{Gyll}. However, this is far from true. The apparent mathematical simplicity of the prey-predator system (usually associated with the classical Rosenzweig-McCarthur model as a paradigm \cite{Rosenzweig63,0c}) is superficial rather than real, and there has recently been a surge of interest and an increase in mathematical modellng literature dealing with its `hidden', overlooked properties, with more than a hundred of papers published in the first quarter of 2021 alone\footnote{Data are taken from the Web of Science}. New properties readily arise as soon as one introduces relatively small (i.e.~preserving the defining structure of the model), biologically motivated changes into the paradigmatic system, e.g.~adding explicit heterogeneous space \cite{Zou20}, changing the specialist predator to a generalist one \cite{Rodrigues20,Sen20}, changing the properties of predator's functional response \cite{Arditi89,Huang14}, considering different types of density dependence in the population growth or mortality \cite{Edwards99,Jiang21}, or taking into account the fact that the intraspecific dynamics of prey and predator often occur on a very different time scale \cite{Kooi18,Poggiale20}.

While one of the generic properties of a prey-predator system is its intrinsic capability to produce sustained population cycles (due to the emergence of a stable limit cycle in a certain parameter range \cite{May72}), with many fundamental implications for the population dynamics, another equally important property is its capacity to exhibit pattern formation, in particualar due to the Turing instability \cite{Segel72,Turing52}. The latter has been a focus of many groundbreaking studies that linked the patterns observed in various biological and ecological systems to the dissipative instability in a prey-predator (or, more generically, activator-inhibitor) system, e.g.~see \cite{Gurney98,Hastings97,Jansen95,McCauley90,Mimura78,Murray68,Murray75,Murray76,Murray81,Murray82,Murray88,Rosenzweig71}, also \cite{Murray89} for an exhaustive review of earlier research. 
Other studies also discovered and considered in detail a possibility of non-Turing pattern formation, in particular due to the interplay between the Hopf bifurcation and diffusion \cite{Pascual93,Petrovskii01,Petrovskii99,Petrovskii02b,Scheffer97,Sherratt95} as well as pattern formation resulting from the Turing-Hopf bifurcation \cite{Baurmann07}.

Interestingly, in spite of the large number of modelling papers concerned with the prey-predator system, there are still a number of issues poorly investigated. One such issue is the interaction between different types of density dependence and the existence of different time scales, either in a spatial or nonspatial system. 
Indeed, one context where the prey-predation framework has been particularly successful to provide a new insight into the mechanisms of ecological interactions are is large-magnitude nearly-periodical fluctuations in population size that has been observed in many species and ecosystems. In such a case, typically, a large outbreak in population abundance is followed by a population decline, often to a small population size or density. For instance, the fluctuations in populations of \textit{snowshoe hares} and \textit{Canadian lynx} in the Canadian Boreal forest was modeled with the help of a tri-trophic food web model \cite{Stenseth97} where the population explosion of the lynx was observed every 9-11 years followed by a rapid decline in the population of hares. Also in case of plankton ecosystem within lake, the seasonal abundance of zooplankton (particularly \textit{Daphnia}) is frequently observed, which completely grazes down the algal biomass thus resulting in clear-water phases in lakes \cite{Scheffer97}. The exact causes of these fluctuations are still a debatable issue among various researchers. However, one of the common trait observed in the above examples is that the bottom level of a multi-trophic system or the basal prey has faster growth and decay compared to their consumers. On the other hand, the population of the budworm can increase several hundred fold within a span of few years whereas the leaves of adult trees do not grow at a comparable rate. This resulted in the outbreak of \textit{spruce budworm} which destroyed the \textit{balsam forest} of eastern Northern America \cite{Ludwig78}. To capture this type of rapid growth/decay for the interacting species, researchers introduced the mathematical models with slow-fast time scale. A small time scale parameter is introduced either in the prey growth equation or in the predator growth depending upon the species under consideration.

In a rather general case, the prey-predator interaction in a nonspatial system can be modeled by a system of coupled ordinary differential equations 
\begin{equation}
\begin{aligned}
u' &= uf(u)-vg(u,v),\\
v' &= evg(u,v) - m(v) v,  
\end{aligned}
\end{equation} 
where $u$ and $v$ denote the prey and predator densities, respectively, at time $t$. (A spatially explicit approach includes diffusion terms, hence turning the ODEs to PDEs, see Section \ref{sec:spatial}). Here both the species are assumed to be distributed homogeneously within their habitat. The function $f(u)$ represents the per capita growth rate of prey, $g(u,v)$ describes the prey-predator interaction and $\mu$ describes the natural death rate of predators in absence of prey, $e$ is known as the conversion efficiency. Assuming the rate of growth of prey population much faster than its predator, a time-scale parameter $0<\varepsilon\ll1$ is introduced which transform the original model to a slow-fast model as follows
\begin{equation}
\begin{aligned}
\varepsilon u' &= uf(u)-vg(u,v),\\
v' &= evg(u,v) - m(v) v.
\end{aligned}
\end{equation} 
This type of slow-fast prey predator models were first studied by Rinaldi and Muratori \cite{Rinaldi92}, so far as our knowledge goes, where the cyclic coexistence of the slow-fast limit cycle was discussed. They also analyzed the cyclic fluctuation in population densities of three species model in a slow-fast setting with one and two multiple time scale parameters. 

In mathematical literature, the slow-fast systems are considered as singularly perturbed ordinary differential equation, where $\varepsilon$ is the singular perturbation parameter. The standard stability and bifurcation analysis performed for the prey-predator models was not enough to analyze the complete dynamics exhibited by the slow-fast systems. Many mathematical techniques were developed to study this class of systems. In the late 1970s, Neil Fenichel \cite{Fenichel79} introduced a geometric approach based on the invariant manifold theory to study the singularly perturbed coupled systems, known as Geometric Singular Perturbation Theory (GSPT). Using this theory the dynamics of the full slow-fast system are studied by reducing it to sub-systems of lower dimension and thereby studying the complete dynamics of the subsystems. The application of Fenichel's theory in the context of biology was well explained by G. Hek \cite{Hek10}. But this theory fails to approximate the dynamics near the non-hyperbolic equilibrium points where the system encounters a singularity. Later in 2001, Krupa and Szmolyan \cite{Krupa01A,Krupa01B} extended Fenichel's theory to overcome the difficulty around non-hyperbolic points using the blow-up technique. This was based on the pioneering work of Dumortier \cite{Dumortier78,Dumortier93,Dumortier96}. The main idea behind this was to blow up the non-hyperbolic equilibrium points of the system by a 4-dimensional unit sphere $S^3$ and the trajectories of the blow-up system is mapped on and around the sphere. In case, the blow-up space still has non-hyperbolic points, sequence of blow-up maps can be used to desingularize the system.

Before the development of mathematical tools to study this class of systems, a Dutch Physicist Van der Pol \cite{Dumortier96,VanderPol26} observed large amplitude periodic oscillation consisting of slow and fast dynamics, which he named relaxation oscillation. These are periodic solutions consisting of slow curvilinear motion and sudden fast jumps. These types of slow-fast limit cycles were later observed in many chemical and biological systems \cite{Kooi18,Muratori89,Rinaldi92,Wang19,Wang19AML}. Another type of periodic solution observed in singularly perturbed systems are canard solutions. This was first investigated by E. Beno\^{i}t {\it et. al} \cite{Benoit81} while studying the Van der Pol Oscillator. Dumortier and Roussarie, in their seminal work \cite{Dumortier96}, analyzed this phenomenon through a geometric approach, using blow-up technique and with the help of invariant manifold theory. A canard is a solution of a singularly perturbed system which follows an attracting slow manifold, closely passing through the bifurcation point of the critical manifold, and then following a repelling slow manifold for $\mathcal{O}(1)$ time. It was observed that for the existence of canard solution Hopf bifurcation is necessary \cite{Krupa01B}. The fast transition from small stable limit cycles appearing through Hopf bifurcation to large amplitude relaxation oscillation via a sequence of canard cycles within an exponentially small range of the parameter is known as canard explosion. In real-world ecosystems this phenomenon can be related to sudden outbreak or decline of a particular species \cite{Ludwig78,Scheffer00,Scheffer97,Siteur16,Stenseth97}.

Over the last few years, several works have been done on prey-predator systems with slow-fast time scale. In \cite{Muratori89,Muratori92,Rinaldi92}, the authors have analyzed the periodic bursting of high and low-frequency oscillations in interacting population models with two and three-trophic level with slow-fast time scale. A novel 1-fast-3-slow dynamical system have been developed in \cite{Piltz17} to consider the adaptive change of diet of a predator population that switches its feeding between two prey populations. The classical Rosenzweig–MacArthur (RM) model and the Mass Balance chemostat model in the slow-fast setting is studied in \cite{Kooi18}, where the authors have shown that the RM model exhibits canard explosion in the oscillatory regime of the parameter space whereas the later model does not exhibit such phenomenon. They have used the asymptotic expansion technique to determine the canard explosion point. In \cite{Poggiale20}, the authors have used the blow-up technique to obtain an analytical expression of the bifurcation thresholds for which maximal canard solution occurs in the RM-model. The existence and uniqueness of the relaxation oscillation cycle have been studied for the Leslie-Gower model with the help of entry-exit function and GSPT in \cite{Wang19AML}. The rich and complex slow-fast dynamics of the predator-prey model with Beddington-DeAngelis functional response is studied in \cite{Saha21}. To the best of our knowledge, there is no work so far in literature considering the Allee effect in the slow-fast prey-predator model. In this paper, first we consider the slow-fast dynamics of the classical Rosenzweig–MacArthur (RM) model with multiplicative weak Allee effect in prey growth equation using GSPT and blow-up technique.

In population ecology, the Allee effect is a widely observed phenomenon especially at low population density, which describes a positive relationship between species population and per capita population growth rate of species. The main causes of Allee effect include difficulties in mate finding, inbreeding depression, cooperative defense mechanism etc. Mostly, we are concerned about the demographic Allee effect which can be classified as: strong Allee and weak Allee effect. For strong Allee effect, the per capita growth rate is negative below some critical population density (Allee threshold), and the growth rate becomes positive above that threshold. Whereas, in case of weak Allee effect, per capita growth rate is small and remains positive even at low population densities. But with the introduction of the time scale parameter the per capita growth rate becomes much higher even at low density. Thus the species can recover itself from the endemic level, and extinction is prevented. In this paper, we have incorporated the weak Allee effect in prey's growth in order to capture the true essence of the slow-fast cycle. 

The main objective of this paper is to provide a detailed slow-fast analysis of the temporal model based on the various mathematical approach discussed above, and numerically investigating the corresponding spatially extended slow-fast model. In prey-predator models, the oscillatory dynamics of the system arises from the Hopf bifurcation but in the slow-fast setting other than Hopf bifurcating limit cycle, the system exhibit various other interesting periodic solution namely canard and relaxation oscillation. Here we are interested to explore these solutions analytically and numerically with the help of sophisticated slow-fast techniques as discussed above. 

Over the last few decades, significant work has been done to study the mechanism of spatial dispersal of species, with the help of reaction-diffusion systems. In this regard, the study of invasion of the exotic species emerged to be of particular interest for many theoretical or field ecologists. Biological invasion is a complex phenomenon that starts with a local introduction of exotic species and once it gets established in a particular region, they start spreading and occupying new areas \cite{Lewis16}. The study of rate and pattern of spread is of primary importance as they have a huge environmental impact and also can be used as a control measure for other species. The invasion of exotic species takes place via propagation of continuous wave fronts, as well as via irregular movement of separate population patches \cite{Lewis00}. In \cite{Morozov06,Petrovskii02a}, the authors have shown that patchy invasion is possible in deterministic models as a result of the Allee effect. To the best of our knowledge, there exist hardly any work in literature which such complete analysis of the slow-fast prey-predator model as well as on the effect of explicit time-scale parameters on pattern formation. Here, first we perform exhaustive numerical simulations to examine the pattern spread and patchy invasion of the species. And then, we examine how the invasion of the species are affected by the time-scale parameter. 

This paper is divided into two parts, in the first part we will provide mathematical analysis of the temporal slow-fast model and the second part is supported by exhaustive numerical simulations to investigate the effect of varying time-scale in biological invasion. In section 2, we introduce the non-dimensionalised temporal model and standard stability analysis is performed. Then in section 3 we introduced the slow-fast system. In section 4 we discussed GSPT and blow-up technique for a detailed mathematical analysis of slow-fast systems. The existence and uniqueness of the relaxation oscillation is studied here followed by the phenomenon of canard explosion. In section 5, we consider the corresponding slow-fast spatio-temporal model to examine how the spread of invasive species is affected by time-scale parameters. Finally, we draw the conclusion of our work in section 6.

\section{Temporal Model and its linear stability analysis}\label{Section:2}
We consider the classical Rosenzweig-MacArthur prey-predator model with the multiplicative weak Allee effect in prey growth \cite{Courchamp08,Murray89,Sen11}. Let $u$ and $v$ be the prey and its specialist predator densities, respectively, at time $t$. In appropriately chosen dimensionless variables and parameter (see \cite{Morozov06} for details), the model is given by the following equations:
\begin{subequations} \label{eq:temp_weak}
	\begin{eqnarray}
	\dfrac{du}{dt} &=& f(u,v):=\gamma u(1-u)(u+\beta)-v\Big(\frac{u}{1+\alpha u}\Big),\\
	\dfrac{dv}{dt} &=& g(u,v):=v\Big(\frac{u}{1+\alpha u}-\delta \Big).
	\end{eqnarray}
\end{subequations}
Here and below, the sign ``:='' means ``is defined''. 
We focus on the case where the growth rate of the prey population is damped by the weak Allee effect, so that $0<\beta<1$. For $\beta<0$, the Allee effect becomes strong (in this case, the prey population has another [usntable] equilibrium at  $u=\beta$); for $\beta\ge 1$, the Allee effect is absent \cite{Lewis93}. The per capita growth rate $f(u,v)/u$ is increasing for $0<u<\dfrac{1-\beta}{2}$ and decreasing for $\dfrac{1-\beta}{2}<u<1$. 
The predator is a specialist predator as they do not have any alternative food source to survive apart from $u$. The prey-dependent functional response is taken to be Holling type II \cite{Holling65}. The system contains four positive dimensionless parameters where $\beta$ quantifies the weak Allee parameter, $\gamma$ is the coefficient
proportional to the maximum per capita growth rate, called characteristic growth rate \cite{Jankovic14}. The parameter $\alpha$ characterizes the inverse saturation level of the functional response and $\delta$ is the natural mortality rate of the predator. Throughout this paper we will consider $\delta$ as the bifurcation parameter to determine the stability conditions of the coexisting steady-state for the model \ref{eq:temp_weak}.

\noindent Depending on the species traits, the prey population often grows much faster than its predator; one well known example is given by hare and lynx where hares reproduce much faster than lynx \cite{Stenseth97}. This motivated researchers to introduce a small time-scale parameter $\varepsilon,\ 0<\varepsilon<1$ in the basic model (\ref{eq:temp_weak}). The parameter $\varepsilon$ is interpreted as the ratio between the linear death rate of the predator and the linear growth rate of the prey \cite{Hek10,Rinaldi92}. And the assumption $\varepsilon < 1$ implies that one generation of predator can encounter several generations of prey \cite{Holling65,Kuehn15}. Therefore considering the difference in the time scale, the slow-fast version of the dimensionless model (\ref{eq:temp_weak}) can be written as
\begin{subequations} \label{eq:temp_weak_fast}
	\begin{eqnarray}
	\dfrac{du}{dt} &=& f(u,v)=\gamma u(1-u)(u+\beta)-\frac{uv}{1+\alpha u},\\
	\dfrac{dv}{dt} &=&\varepsilon g(u,v)=\varepsilon v \Big(\frac{u}{1+\alpha u}-\delta \Big),
	\end{eqnarray}
\end{subequations}
with initial conditions $u(0)\ge 0,\ v(0)\ge 0$ . Since the prey population grows faster compared to the predator, $u$ and $v$ are referred to as fast and slow variables, respectively and time $t$ is called fast time. The equilibrium points for the system  are independent of $\varepsilon$, thus system (\ref{eq:temp_weak}) and (\ref{eq:temp_weak_fast}) has same equilibrium points. The extinction equilibrium point and prey only equilibrium point of system (\ref{eq:temp_weak}) (as well as for (\ref{eq:temp_weak_fast}) are given by $E_0=(0,0)$ and $E_1=(1,0)$ respectively. The interior equilibrium point $E_*(u_*,v_*) $ of the system is the point where the non-trivial prey nullcline intersect with non-trivial predator nullcline in the interior of the positive quadrant, and we have,
$$u_* = \dfrac{\delta}{1-\alpha \delta}, \ \ v_* = \gamma (1-u_*)(u_*+\beta)(1+\alpha u_*).$$
$E_*$ is feasible if the parametric restriction $\delta (\alpha + 1)<1$ holds. With the help of linear stability analysis, we find $E_0$ is always a saddle point. $E_1$ is stable for $\delta > \dfrac{1}{1+\alpha}$ and saddle point for $\delta < \dfrac{1}{1+\alpha}$. $E_*$ bifurcates from predator free equilibrium point $E_1$ through transcritical bifurcation at $\delta = \delta_T\equiv \dfrac{1}{1+\alpha}.$\\ Now evaluating the Jacobian matrix  for the system (\ref{eq:temp_weak_fast}) at the interior equilibrium point $E_*(u_*,v_*)$ we have 
\begin{equation*}\label{eq:Jacobian matrix}
J_*=
\begin{pmatrix}
\gamma(u_*(2-3u-2\beta)+\beta)-\dfrac{v_*}{(1+u_*\alpha)^2}&-\dfrac{u_*}{1+u_*\alpha}\\\dfrac{\varepsilon v_*}{(1+u_*\alpha)^2}&\varepsilon\Big(\dfrac{u_*}{1+u_*\alpha}-\delta\Big)
\end{pmatrix}.
\end{equation*}
From the feasibility condition of $E_*$ we always have Det$(J_*)>0$. The interior equilibrium point is stable if Tr$(J_*)<0$, and it loses its stability via super-critical Hopf bifurcation when Tr$(J_*)=0$ and is unstable for Tr$(J_*)>0$.
The Hopf threshold $\delta=\delta_H$ can be obtained by solving Tr$(J_*)=0$ which on simplification gives
\begin{equation}\label{eq:Hopf_threshold}
\delta_H = \dfrac{1+\alpha^2\beta-\sqrt{1+\alpha+\alpha^2-\alpha \beta +\alpha^2\beta+\alpha^2 \beta^2}}{\alpha(-1-\alpha+\alpha \beta +\alpha^2\beta)}.
\end{equation}
Transversality condition for Hopf bifurcation is satisfied at $\delta=\delta_H.$
The coexistence steady state $E_*(u_*,v_*)$ is stable for $\delta>\delta_H$ and it destabilizes for $\delta<\delta_H$, surrounded by a stable limit cycle. The bifurcation diagrams of the system (\ref{eq:temp_weak}) with $\delta$ as bifurcation parameter and for two different values of $\beta$ are plotted in Fig.~\ref{fig:bifurcation_eps_1}. 
It is readily seen that, in case $\beta<1/\alpha$, with an increase in the Allee threshold $\beta$ (hence making the Allee effect weaker) the Hopf bifurcation point shifts to the left. Correspondingly, the steady species coexistence occurs for a broader parameter range of $\delta$ (see Fig.~\ref{fig:bifurcation_eps_1}). Furthermore, with the increase of $\beta$ the size of the limit cycle is reduced. Increase in the strength of the weak Allee effect not only enhance the stable coexistence rather reduces the amplitude of stable oscillatory coexistence.
\begin{figure}[!ht]
	\centering
	\mbox{\subfigure[$\beta=0.2$]{\includegraphics[width=0.45\textwidth]{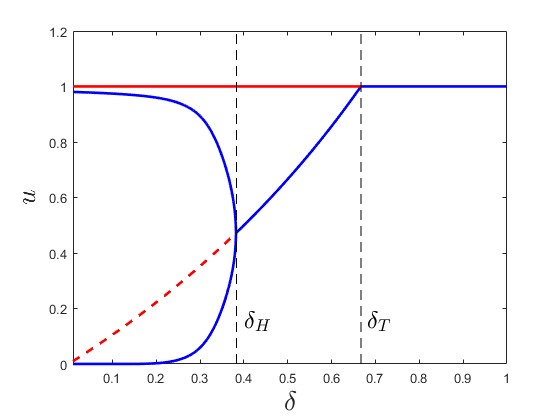}}
		\subfigure[$\beta=0.8$]{\includegraphics[width=0.45\textwidth]{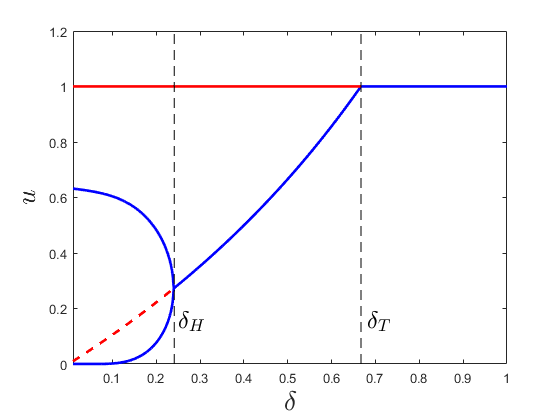}}
	}
	\caption{The bifurcation diagram of system (\ref{eq:temp_weak}) using $\delta$ as the bifurcation parameter shown for two different values of $\beta$ and other parameters as \ $\alpha=0.5,\ \gamma=3.$ Here the blue solid curve shows either the stable steady state (for $\delta>\delta_H$) or the size of the stable limit cycle (for $\delta<\delta_H$), the solid red line shows the semitrivial `prey-only' state in the range where it is a saddle (hence unstable), and the dashed red line shows the unstable coexistence steady state.}\label{fig:bifurcation_eps_1}
\end{figure}
\noindent Interestingly, the linear stability results remains unaltered in the presence of slow-fast time scale as the analytical conditions are independent of $\varepsilon$. The linear stability analysis fails to capture the complete dynamics of the slow-fast system (\ref{eq:temp_weak_fast}) for $0<\varepsilon \ll 1$. The system (\ref{eq:temp_weak_fast}) exhibit catastrophic transition which cannot be captured by standard stability analysis, rather the model may sometimes overestimate ecological resilience \cite{Siteur16}. Therefore, to study the complete dynamics of the system we take help of geometric singular perturbation theory and blow-up technique which will be discussed in next sections.

\section{Slow-fast system}

In this section, we shall describe the dynamics of the slow-fast system (\ref{eq:temp_weak_fast}) when $0<\varepsilon\ll1$. To understand the dynamics of the system (\ref{eq:temp_weak_fast}) for sufficiently small $\varepsilon\ (>0)$ we need to consider the behaviour of two subsystems corresponding to (\ref{eq:temp_weak_fast}), which can be obtained for $\varepsilon=0$. The system in its singular limit, $\varepsilon=0$ is obtained as follows
\begin{subequations} \label{eq:layer-system}
	\begin{eqnarray}
	\dfrac{du}{dt} &=& f(u,v)= \gamma u(1-u)(u+\beta)-\frac{uv}{1+\alpha u},\\
	\dfrac{dv}{dt} &=&0.
	\end{eqnarray}
\end{subequations}
The above system is known as fast subsystem or layer system corresponding to the slow-fast system (\ref{eq:temp_weak_fast}). The fast flow consists with constant predator density determined by the initial condition $v(0)=c$, and by integrating the differential equation 
\begin{equation} \label{eq:fast-flow}
\begin{aligned}
\dfrac{du}{dt} = \gamma u(1-u)(u+\beta)-\frac{uc}{1+\alpha u},
\end{aligned}
\end{equation}
with initial condition $u(0)>0$. The direction of the fast flow depends on the choice of initial conditions $u(0)$, $v(0)$ and other parameter values. Green horizontal lines are solution trajectories with appropriate direction as shown in Fig.~\ref{fig:dynamics_slow_fast}a. 
\noindent Now writing system (\ref{eq:temp_weak_fast}) in terms of the slow time $\tau:=\varepsilon t$, we get the equivalent system in terms of slow time derivatives,
\begin{subequations} \label{eq:temp_weak_slow}
	\begin{eqnarray}
	\varepsilon \dfrac{du}{d\tau} &=& f(u,v) = \gamma u(1-u)(u+\beta)-\frac{uv}{1+\alpha u},\\
	\dfrac{dv}{d\tau} &=& g(u,v) =v\Big(\frac{u}{1+\alpha u}-\delta \Big).
	\end{eqnarray}
\end{subequations}
Substituting $\varepsilon = 0$ in the above system to find the following differential algebraic equation (DAE),
\begin{subequations} \label{eq:DAE}
	\begin{eqnarray}
	0&=& f(u,v) = \gamma u(1-u)(u+\beta)- \frac{uv}{1+\alpha u},\\
	\dfrac{dv}{d\tau} &=& g(u,v) = v\Big(\frac{u}{1+\alpha u}-\delta \Big),
	\end{eqnarray}
\end{subequations}
which is known as the slow subsystem corresponding to the slow-fast system (\ref{eq:temp_weak_slow}). The solution of the above system is constrained to the set $\{(u,v) \in \mathbb{R}^2_+:f(u,v)=0\}$ and is known as critical manifold $C_0$. This set has one-one correspondence with the set of equilibrium of the system (\ref{eq:fast-flow}). This critical manifold consists of two of different manifolds
\begin{equation*}
\begin{aligned}
C_0^0 &=&  &\{(u,v)\in \mathbb{R}^2_+ : u=0, v\ge0\},\\
C_0^1 &=& &\Big\{(u,v) \in \mathbb{R}^2_+ :v =q(u):= \gamma (1-u)(u+\beta)(1+\alpha u),\ 0<u<1,\ v > 0\Big\},
\end{aligned}
\end{equation*}
such that $C_0 = C^0_0 \cup C^1_0$ where $C_0^0$ is the positive $v$-axis and $C_0^1$ is a portion of a parabola as shown in Fig.~\ref{fig:dynamics_slow_fast}a, marked with black colour. The slow flow on the critical manifold is given by
\begin{equation} \label{eq:slow-flow}
\begin{aligned}
\dfrac{du}{d\tau} &=& \dfrac{g(u,q(u))}{\dot{q}(u)},
\end{aligned}
\end{equation}
where $`.`$ refers to the differentiation w.r.t $u$.
\noindent The solution of the system (\ref{eq:temp_weak_fast}) for sufficiently small $\varepsilon >0$ cannot be approximated from its limiting solution at $\varepsilon=0$. Therefore, $\varepsilon=0$ is the singular limit of the system (\ref{eq:temp_weak_fast}). The solution of the full system is obtained by combining the solution of the system in its singular limits. And depending on the region in the phase space we use either of the subsystems.

\noindent For $\alpha=0.5$, $\beta=0.22$, $\gamma=3$, and $\delta=0.3$ the coexistence steady-state is unstable for $0<\varepsilon\leq1$ and is surrounded by a stable limit cycle. Interestingly the size and shape of stable limit cycle change with the variation of $\varepsilon$ which is shown in Fig. \ref{fig:dynamics_slow_fast}(b). The size and shape of closed curve attractor (blue), obtained for $\varepsilon=0.001$ is quite different from stable limit cycle (magenta) which is obtained for $\varepsilon=1$. This shape does not change if we further decrease $\varepsilon$, keeping other parameters fixed. This observation is based upon the numerical simulation and we need detailed analysis to understand the possible shape of the trajectories in singular limit $\epsilon\rightarrow0$. For $\varepsilon=0.001$, the closed curve attractor (blue) consists of two horizontal segments on which flow is fast and one curvilinear and vertical segment where the flow is slow. They are obtained by concatenating the solution of the layer and the reduced subsystem respectively. The two horizontal segments of the attractor (blue) are the
perturbed trajectories corresponding to the layer system. This signifies the fast growth or decay of the prey species while predator density remains unaltered. The vertical portion close to $v$-axis and curvilinear part are close to the critical manifolds $C^0_0$ and $C^1_0$.  Change in shape and size of the attractor does not solely depend upon the magnitude of $\varepsilon$ rather determined by the magnitude of the parameters involved with the reaction kinetics and time scale parameter. 

\noindent Now we fix $\varepsilon=0.01$ and other parameters as mentioned above, except $\delta$. Small variation in $\delta$ just below the  $\delta_H$ results in rapid change in size and shape of the periodic attractor (see Fig.~\ref{fig:dynamics_slow_fast}(c)). A small limit cycle (cyan) appears for $\delta=0.3762$ known as canard cycle without head. This cycle encounter a change in curvature when $\delta=0.376165$ and the resulting cycle is known as canard cycle with head (blue). Further decreasing $\delta$ to $0.36$ the system settles down to a closed cycle known as relaxation oscillation. Further decrease in $\delta$ does not alter the size and shape of the closed attractor and the trajectories converge to the stable relaxation oscillation cycle even for $\varepsilon$ sufficiently small. In the next section, we will derive the analytical conditions for the existence of canard cycle and relaxation oscillation. The analytical results will help us to identify the domains in the prarametric plane where we can find these different types of closed curve attractors.
\begin{figure}[!ht]
	\centering
	\mbox{\subfigure[]{\includegraphics[width=0.45\textwidth]{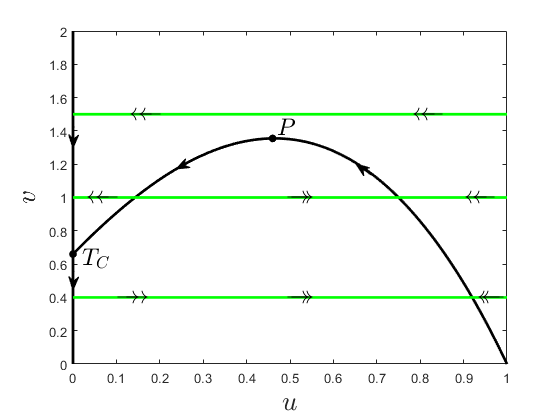}}
		\subfigure[]{\includegraphics[width=0.45\textwidth]{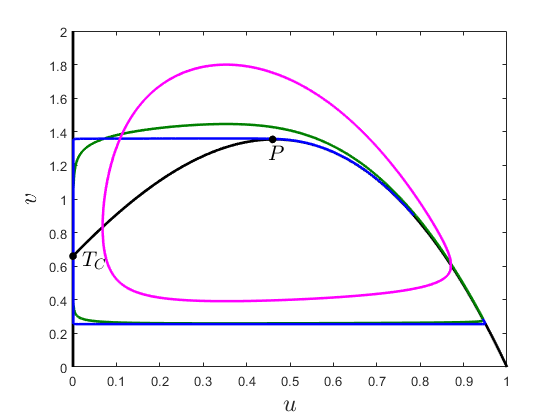}}
	}
	\mbox{\subfigure[]{\includegraphics[width=0.45\textwidth]{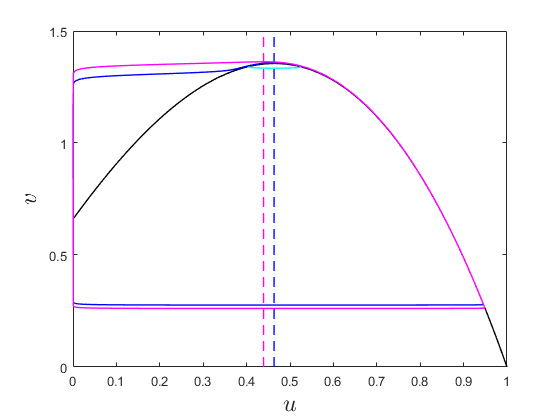}}}
	\caption{(a) Dynamics of the slow-fast system (\ref{eq:temp_weak_fast}) where single arrow represnt slow flow and double arrow represent fast flow, (b) limit cycles for different values of time-scale parameter  $\varepsilon=1$ (magenta), $\varepsilon=0.1$ (green), $\varepsilon=0.001$ (blue), (c) canard cycles for different values of the bifurcation parameter $\delta$ for $\varepsilon=0.01$, canard cycle without head for $\delta=0.3762$ (cyan), canard cycle with head for $\delta=0.376165$ (blue), relaxation oscillation for $\delta=0.36$ (magenta) and other parameters are mentioned in the text.  }
	\label{fig:dynamics_slow_fast}
\end{figure}

\section{Analysis of slow-fast system}

The critical manifold $C_0^1$ can be divided into two parts, one part consists of the attractors of the fast sub-system and another part is repelling in nature. The attracting and repelling part of the manifold is separated by a non-degenerate fold point $P$. The fold point $P(u_f,v_f)$ is characterized by the following conditions \cite{Krupa01A}.	
$$ \frac{\partial{f}}{\partial u}(u_f,v_f)=0,\ \frac{\partial f}{\partial v}(u_f,v_f)\ne 0,\ \frac{\partial^2f}{\partial u^2}(u_f,v_f)\ne 0,\ \text{and}\ g(u_f,v_f)\ne 0.$$ 
The components of the fold point are given by $$u_f = \dfrac{(\alpha-\alpha \beta -1)+(1+\alpha+\alpha^2-\alpha \beta+\alpha^2\beta+\alpha^2 \beta^2)^{\frac{1}{2}}}{3\alpha}, $$ $$ v_f = \gamma (1-u_f)(u_f+\beta)(1+\alpha u_f), $$
which is the maximum point on the critical manifold. The fold point divides the critical manifold into normally hyperbolic attracting ($C^{1,a}_0$) and repelling ($C^{1,r}_0$) submanifolds given by
\begin{equation*}
\begin{aligned}
C_0^{1,a} &=\Big\{(u,v) \in \mathbb{R}^2_+ :v =q(u), u_f<u\le1\Big\} \\
C_0^{1,r} &=\Big\{(u,v) \in \mathbb{R}^2_+ :v =q(u), 0\le u<u_f\Big\}.
\end{aligned}  
\end{equation*}
The point of intersection of $C^1_0$ with the vertical $v$-axis is $T_C(0,\beta \gamma)$, which is the transcritical bifurcation point for the fast subsystem. It follows from Fenichel's theorem \cite{Fenichel79,Kuehn15}, there exist locally invariant slow sub-manifolds $C^1_{\varepsilon}$ and $C_{\varepsilon}^0$ which are diffeomorphic to the respective critical manifolds $C^1_0$ and $C^0_0$, except at the non-hyperbolic points $P$ and $T$. $C^1_0$ can be written explicitly as $v=q(u)$, we assume the invariant manifold $C_{\varepsilon}^1$ can be written as a perturbation of $v=q(u)$ as follows, with $\varepsilon$ as perturbation parameter,
\begin{equation} \label{eq:invariant-manifold}
C_{\varepsilon}^1=\Big\{(u,v) \in \mathbb{R}^2_+ :v =q(u,\varepsilon),\ 0< u<1,\ v >0\Big\},
\end{equation}
where
$q(u,\varepsilon) = q_0(u) + \varepsilon q_1(u) + \varepsilon^2 q_2(u)+\cdots$, and
$C^0_{\varepsilon}=\{(u,v)\in \mathbb{R}^2_+ : u=0, v\ge0\}$.
Using the invariance condition and the asymptotic expansion of $q(u,\varepsilon)$, we can find the perturbed invariant manifold approximated up to the desired order. The approximation of $q(u,\varepsilon)$ up to second order is provided in Appendix A with explicit expressions for $q_0$,  $q_1$, $q_2$. The approximations of invariant manifold for different values of $\varepsilon$ are shown in Fig.~\ref{fig:perturbed_manifolds}. This approximation has two non-removable discontinuities in the vicinity of the non-hyperbolic points $P$ and $T_C$.

\noindent The critical manifold $C^1_0$ is normally hyperbolic except at the points $P (u_f,v_f)$ and $T_C(0,\beta \gamma)$ and so is $C^1_\varepsilon$. Thus any trajectory starting near the attracting (repelling) submanifold $C^{1,a}_0$ $(C^{1,r}_0)$ cannot cross the fold point $P$ (transcritical point $T_C$). We can see from Fig.~\ref{fig:dynamics_slow_fast} that for sufficiently small values of $\varepsilon$ the trajectories pass enough close to the attracting manifold $C^{1,a}_0$ and cross the point $P$. Fenichel's theory is not adequate to determine the analytical expression for perturbed sub-manifolds close to $C_0^{1}$ and is continuous in the vicinity of the  non-hyperbolic points.
\begin{figure}[!ht]
	\begin{center}
		\mbox{
			\includegraphics[width=0.45\textwidth]{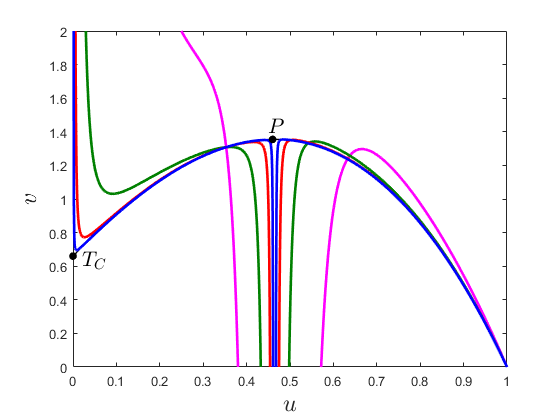}}
		\caption{ Second order approximation of perturbed manifold from GSPT with $\varepsilon=1$ (magenta), $\varepsilon=0.1$(green),\ $\varepsilon=0.01$(red),\ $\varepsilon=0.001$(blue); for the parameter values $\alpha=0.5,\ \beta=0.22,\ \gamma=3,\ \delta=0.3$ .}\label{fig:perturbed_manifolds}
	\end{center}
\end{figure}

\noindent Therefore, to construct a trajectory passing through the vicinity of the point $P$ we must remove the singularity at this point. Depending on parameter $\delta,$ the predator nullcline intersects either $C^{1,a}_0$ or $C_0^{1,r}$ or passes through the point $P$. Thus the coexistence equilibrium point $E_*$ lies either on $C^{1,a}_0$ or $C_0^{1,r}$ or coincides with $P$.  When $E_*$ lie on $C_0^{1,a}$ it is globally asymptotically stable and every trajectory converges to $E_*$. When $E_*$ coincides with the fold point $P$ then  $ f_u(E_*)=0$, $f_v(E_*)\ne 0$, $f_{uu}(E_*)\ne 0$, and  $g(E_*)= 0$, this point is called the canard point. For the model under consideration (\ref{eq:temp_weak_fast}), the Hopf point coincides with the canard point. The solution passing through this point is known as the canard solution. For $\delta<\delta_H$, the coexistence equilibrium point $E_*$ lies on the repelling sub-manifold $C_0^{1,r}$, it is unstable and we obtain a special kind of periodic solution consisting of two fast flow (almost horizontal) and two slow flow (passing close to $C^{1,a}_0$ and $C^0_0$) called relaxation oscillation. We will discuss the existence of such solutions in consequent subsections.

\noindent To remove the singularity at the fold point, we use blow-up transformation at the non-hyperbolic fold point which will extend the system over a 3-sphere in $\mathbb{R}^4$, denoted by $S^3 = \{x\in \mathbb{R}^4: ||x||=1\}$. Using the blow-up technique we remove the singularity from the system and determine the canard solution passing through this point. 

\noindent To apply the blow-up technique, first, we transform the slow-fast system (\ref{eq:temp_weak_fast}) into its desired slow-fast normal form.

\subsection{Slow-fast normal form}

Here we consider a topologically equivalent form of the system (\ref{eq:temp_weak_fast}) by re-scaling the time with the help of a transformation $t \rightarrow (1+\alpha u)t$,  where $(1+\alpha u) >0$ \cite{Kuznetsov04}. The transformed system is
\begin{equation}\label{eq:topo_eq}
\begin{aligned}
\dfrac{du}{dt} &= \gamma u(1-u)(u+\beta)(1+\alpha u)-uv \equiv F(u,v,\delta),\\
\dfrac{dv}{dt} &= \varepsilon\left(uv-\delta v(1+\alpha u)\right) \equiv \varepsilon G(u,v,\delta).
\end{aligned}
\end{equation}
The fold point $P$ coincides with the coexistence equilibrium point at $\delta=\delta_*$. As a consequence, the following conditions hold
\begin{equation}\label{eq:folded_singularity}
\begin{aligned}
&F(u_*,v_*,\delta_*)=0,\ F_u(u_*,v_*,\delta_*)= 0,\ F_v(u_*,v_*,\delta_*)\ne 0,\ F_{uu}(u_*,v_*,\delta_*)\ne 0,\\
&\ G_u(u_*,v_*,\delta_*)\ne 0,\ G_{\delta}(u_*,v_*,\delta_*)\ne 0 \ \text{and} \ G(u_*,v_*,\delta_*)=0.
\end{aligned}
\end{equation} 
Using the transformation $U=u-u_*,\ V=v-v_*,\ \lambda=\delta-\delta_*$, we translate the fold point to the origin, and together with the conditions (\ref{eq:folded_singularity}) the system reduces to the slow-fast normal form near $(0,0)$ as follows
\begin{equation}\label{eq:slow-fast_normal_form}
\begin{aligned}
\dfrac{dU}{dt} &=-Vh_1(U,V)+U^2h_2(U,V)+\varepsilon h_3(U,V),\\
\dfrac{dV}{dt}  &=\varepsilon\left(Uh_4(U,V)-\lambda h_5(U,V)+Vh_6(U,V)\right),
\end{aligned}
\end{equation}
where $h_i's,\ i=1,2,3\cdots 6$ are given in Appendix B. 

Here $\lambda$ measures the perturbation of $\delta$ from $\delta_*$ and is considered as bifurcation parameter for the system (\ref{eq:slow-fast_normal_form}). The bifurcation parameter $\lambda$ and time-scale parameter $\varepsilon$ are assumed to be independent of time. We now extend the above system to $\mathbb{R}^4$ by augmenting the equations $\dfrac{d\lambda}{dt}=0$ and $\dfrac{d\varepsilon}{dt} =0$ to system (\ref{eq:slow-fast_normal_form}) and study the dynamics of the system in the vicinity of $(0,0,0,0)$.

\subsection{Blow-up transformation}
The fold point $P$ of the system (\ref{eq:temp_weak_fast}) and the equilibrium point $E_*$ coincides at the Hopf bifurcation threshold. Hence $P$ is now a Canard point. Let us consider the blow-up space $S^3 = \{(\bar{U},\bar{V},\bar{\lambda},\bar{\varepsilon})\in \mathbb{R}^4: \bar{U}^2+\bar{V}^2+\bar{\lambda}^2+\bar{\varepsilon}^2=1\}$ and an interval $\mathcal{I}:=[0,\rho]$ where $\rho>0$ is a small constant. We define a manifold $\mathcal{M}:=S^3\times\mathcal{I}$ and the blow-up map $\Phi$, $\Phi: \mathcal{M} \rightarrow \mathbb{R}^4$ where 
\begin{equation}\label{eq:blow-up-map}
\Phi(\bar{U},\bar{V},\bar{\lambda},\bar{\varepsilon},\bar{r}) = (\bar{r}\bar{U},\bar{r}^2\bar{V},\bar{r}\bar{\lambda},\bar{r}^2\bar{\varepsilon}):=(U,V,\lambda,\varepsilon).
\end{equation}
Using the above map we can write the transformed system as follows 
\begin{equation}\label{eq:blow_up_system}
\begin{aligned}
\dfrac{d\bar{U}}{dt} &= \dfrac{1}{\bar{r}}\left( \dfrac{dU}{dt} - \bar{U}\dfrac{d\bar{r}}{dt}\right),\
\dfrac{d\bar{V}}{dt} = \dfrac{1}{\bar{r}^2}\left(\dfrac{dV}{dt} - 2\bar{r} \bar{V}\dfrac{d\bar{r}}{dt}\right),\\
\dfrac{d\bar{\lambda}}{dt} &= \dfrac{1}{\bar{r}}\left( \dfrac{d\lambda}{dt} - \bar{\lambda}\dfrac{d\bar{r}}{dt}\right),\
\dfrac{d\bar{\varepsilon}}{dt} = \dfrac{1}{\bar{r}^2}\left(\dfrac{d\varepsilon}{dt} - 2\bar{r} \bar{\varepsilon}\dfrac{d\bar{r}}{dt}\right),
\end{aligned}
\end{equation}
where $\dfrac{d\bar{U}}{dt},\dfrac{d\bar{V}}{dt}, \dfrac{d\bar{\lambda}}{dt}, \dfrac{d\bar{\varepsilon}}{dt}$ are given in (\ref{eq:slow-fast_normal_form}). To study the dynamics of the transformed system on and around the  hemisphere $S^3_{\varepsilon\ge 0}$ we will introduce the charts with direction blow-up maps \cite{Kuehn15,Tu}. Along each direction of the coordinate axis we define the charts $K_1$, $K_2$, $K_3$ and $K_4$ by setting $\bar{V}=1$, $\bar{\varepsilon}=1$, $\bar{U}=1$ and $\bar{\lambda}=1$ respectively, in (\ref{eq:blow_up_system}). The charts $K_1$ and $K_3$ describe the dynamics in the neighborhood of the equator of $S^3$ and $K_2$ describes the dynamics in a neighborhood of the positive hemisphere. Here, we mainly focus on chart $K_2$ to prove the existence of a periodic solution for $0<\bar{\varepsilon}\ll1$. Re-scaling the time with the transformation $\bar{t}:=\bar{r}t$, we desingularize the system (\ref{eq:blow_up_system}) so that the multiplicative factor $\bar{r}$ disappears. The transformed version of the system (\ref{eq:blow_up_system}) can be written in chart $K_2$ as follows
\begin{equation}\label{eq:desingularized_system_K2}
\begin{aligned}
\dfrac{d\bar{U}}{dt} &=-\bar{V}b_1+\bar{U}^2b_2+\bar{r}\left(a_1\bar{U}-a_2\bar{U}\bar{V}+a_3\bar{U}^3\right)+O(\bar{r}(\bar{\lambda}+\bar{r})),\\
\dfrac{d\bar{V}}{dt} &= \bar{U}b_3-\bar{\lambda}b_4 + \bar{r}\left(a_4\bar{U}^2+a_5\bar{V}\right)+O(\bar{r}(\bar{\lambda}+\bar{r}),\\
\dfrac{d\bar{\lambda}}{dt} &= 0,\\
\dfrac{d\bar{\varepsilon}}{dt} &=0,
\end{aligned}
\end{equation}
where $b_j$'s are given in Appendix B,  
\begin{equation}\label{eq:a_i}
\ a_1 = a_5 = 0,\ \
a_2 = 1,\ \ 
a_3 = -\left(\gamma+\alpha \gamma(4u_*+\beta-1)\right),\ \
a_5 = u_*-(1+u_*\alpha)\delta_*.
\end{equation}

\noindent The condition for the destabilization of the coexistence equilibrium point through singular Hopf-bifurcation is summarized in the following theorem. The Hopf bifurcation of the system (\ref{eq:slow-fast_normal_form}) occurs at $\lambda=0$. At Hopf bifurcation, the purely imaginary eigenvalues of the corresponding Jacobian matrix are a function of $\varepsilon$, which tends to zero as $\varepsilon \rightarrow 0$. Also, in slow-time derivative, the eigenvalues are a function of $1/\varepsilon. $ Thus on both the time scales the Hopf bifurcation is singular as $\varepsilon \rightarrow 0.$ 

\begin{theorem}
	Let $(U,V)=(0,0)$ is the canard point of the transformed system (\ref{eq:slow-fast_normal_form}) at $\lambda=0$ such that $(0,0)$ is a folded singularity and $G(0,0,0)=0$. Then for sufficiently small $\varepsilon$ there exist a singular Hopf bifurcation curve $\lambda=\mathcal{\lambda_H}(\sqrt{\varepsilon})$ such that the equilibrium point $p$ of the system (\ref{eq:slow-fast_normal_form}) is stable for $\lambda>\mathcal{\lambda_H}(\sqrt{\varepsilon})$ and 
	\begin{equation}
	\mathcal{\lambda_H}(\sqrt{\varepsilon})=-\dfrac{b_3(a_1+a_5)}{2b_2b_4}\varepsilon+O(\varepsilon^{3/2}).
	\end{equation}
\end{theorem}
\begin{proof} 
	The proof of the theorem is given in Appendix B.
\end{proof}	
\noindent The singular Hopf bifurcation curve for the system (\ref{eq:topo_eq}) is thus given by 
\begin{equation} \label{eq:Singular_Hopf_curve}
\begin{aligned}
\mathcal{\delta_H}(\sqrt{\varepsilon}) = & \dfrac{1+\alpha^2\beta-\sqrt{1+\alpha+\alpha^2-\alpha\beta+\alpha^2\beta+\alpha^2\beta^2}}{\alpha(-1-\alpha+\alpha\beta+\alpha^2\beta)}- \\ & \quad \dfrac{b_3(a_1+a_5)}{2b_2b_4}\varepsilon+O(\varepsilon^{3/2}).
\end{aligned}
\end{equation}
In Fig.~\ref{fig:Singular_Hopf_Maximal_Canard} the singular Hopf bifurcation curve (red) is plotted in $\delta-\varepsilon$ parametric plane, it clearly explains how the singular Hopf-bifurcation threshold changes with the variation in $\varepsilon$. 
Once the coexistence equilibrium loses stability through Hopf bifurcation, at the Canard point, we find a closed orbit as attractor surrounding the unstable equilibrium point. From this point small amplitude stable canard cycle originates enclosing the point $P$ and then forms canard cycle with head depending on the parameter values. The following theorem  provides an anlaytical expression of the maximal canard curve in $(\lambda-\varepsilon)$ plane.

\begin{theorem}
	Let $(U,V)=(0,0)$ is the canard point of the slow-fast normal form (\ref{eq:slow-fast_normal_form}) at $\lambda=0$ such that $(0,0)$ is a folded singularity and $G(0,0,0)=0$. Then for $\varepsilon>0$ sufficiently small there exists maximal canard curve $\lambda=\lambda_c(\sqrt{\varepsilon})$ such that the slow flow on the normally hyperbolic invariant submanifolds $\mathcal{M}_{\varepsilon}^{1,a}$ connects with  $\mathcal{M}_{\varepsilon}^{1,r}$ in the blow-up space. And $\lambda_c(\sqrt{\varepsilon})$ is given by
	\begin{equation}
	\begin{aligned}
	\lambda_c(\sqrt{\varepsilon}) &= -\dfrac{1}{A_5}\Big(\dfrac{3A_1}{4A_4^2}+\dfrac{A_2}{2A_4}+A_3\Big)\varepsilon + O(\varepsilon^{3/2})
	\end{aligned}
	\end{equation}
\end{theorem}

\begin{proof} The proof of this theorem is given in Appendix C.
\end{proof}

\noindent The maximal canard curve, along which the canard cycle with head appears for the system (\ref{eq:topo_eq}) is given by 
\begin{equation}\label{eq:Maximal_canard_curve}
\begin{aligned}
\delta_c(\sqrt{\varepsilon})  = & \dfrac{1+\alpha^2\beta-\sqrt{1+\alpha+\alpha^2-\alpha\beta+\alpha^2\beta+\alpha^2\beta^2}}{\alpha(-1-\alpha+\alpha\beta+\alpha^2\beta)} - \\ & \quad \dfrac{1}{A_5}\Big(\dfrac{3A_1}{4A_4^2}+\dfrac{A_2}{2A_4}+A_3\Big)\varepsilon + O(\varepsilon^{3/2}).
\end{aligned}
\end{equation}
Keeping $\alpha$, $\beta$ and $\varepsilon\ (>0)$ fixed, $\delta_c(\sqrt{\varepsilon})$ gives the threshold for the existence of canard cycle with head. A schematic diagram of the threshold curves in $\delta-\varepsilon$-plane is illustrated in Fig.~\ref{fig:Singular_Hopf_Maximal_Canard} and it divides the $\delta-\varepsilon$ parametric plane into four domains.
\begin{figure}[!ht]
	\begin{center}
		\centering
		\mbox{{\includegraphics[width=0.5\textwidth]{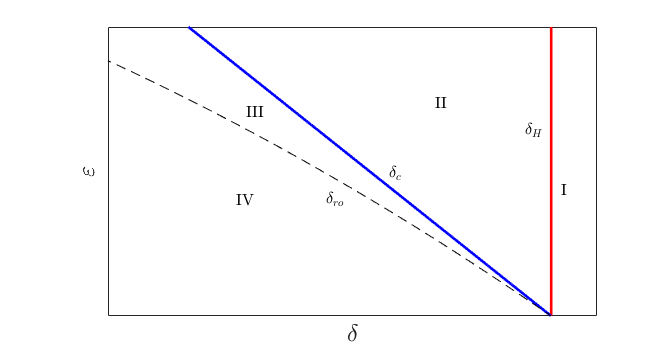}}
		}\caption{Schematic diagram showing singular Hopf bifurcation curve $\mathcal{\delta_H}$ (red), maximal canard curve $\delta_c$ (blue), and  $\delta_{ro}$ (dashed black) relaxation oscillation cycle. }\label{fig:Singular_Hopf_Maximal_Canard}
	\end{center}
\end{figure}
In domain I, when $\delta > \delta_\mathcal{H}$, the coexistence equilibrium point is stable.  For a fixed $\varepsilon>0$, as we decrease $\delta$ from domain I to domain II small amplitude canard cycles appear after crossing the Hopf bifurcation threshold $\delta = \delta_\mathcal{H}.$ In domain II, that is when $\delta_c <\delta < \delta_\mathcal{H},$ the system experiences a transition from canard cycle with head to canard cycle without head. The size of the canard cycle increases on decreasing $\delta$ and the shape of the cycle changes to canard with a head at $\delta=\delta_c.$ The canard cycle whith head persists in a narrow domain III, where $\delta_{ro}<\delta<\delta_c$. On further decreasing $\delta,$ that is, when $\delta\le\delta_{ro}$ the unstable equilibrium point is surrounded by a stable periodic attractor called relaxation oscillation. This periodic attractor consists of two concatenated slow (close to the critical manifold) and fast (almost horizontal and away from the critical manifold) flow.
We can see that for sufficiently small $\varepsilon$, this transition, from small canard cycle to relaxation oscillation through canard cycle with head, takes place within a narrow interval of the parameter $\delta$ and the phenomenon is known as canard explosion. This mechanism is further illustrated with the help of numerical example in the sub-section \ref{sec:Canard_explosion}.

\subsection{Entry-Exit Function} \label{subsec:entry_exit} 
The canard cycle and relaxation oscillation pass through the fold point $P$ when $\varepsilon=0$. These closed attractors pass through the vicinity of $P$ for $\varepsilon\ll1$. We now prove the existence of a trajectory that jumps from the fold point to the other attracting slow manifold $C^0_0$ through fast horizontal flow and continuing there for a constant time the trajectory leaves $C^0_0$ at a certain point. This is determined by the entry-exit function and we can find the coordinates of the exit point from the slow manifold $C^0_0$. To do this, first, rewrite the system (\ref{eq:temp_weak_fast}) as a Kolmogorov system \cite{Freedman80} as follows 
\begin{equation}\label{eq:entry_exit}
\begin{aligned}
\dfrac{du}{dt} &=& uf_1(u,v) &= u\Big(\gamma (1-u)(u+\beta)-\frac{v}{1+\alpha u}\Big),\\
\dfrac{dv}{dt} &=& \varepsilon vg_1(u,v) &= \varepsilon v \Big(\frac{u}{1+\alpha u}-\delta \Big).
\end{aligned}
\end{equation}
We can verify $f_1(0,v)=\gamma \beta-v$, $g_1(0,v)=-\delta<0$, which implies that $f_1(0,v)<0$ if $v> \gamma \beta$, $f_1(0,v)>0$ if  $v< \gamma \beta$. 
$T\ (0,\beta \gamma)$ on the vertical axis is the transcritical bifurcation point and we can divide the slow manifold $C^0_0$ into two parts $V^+:=\{(u,v):u=0, v>\gamma \beta\}$ and $V^-:=\{(u,v):u=0,\ v<\gamma \beta\}$. Clearly $V^+$ is attracting and $V^-$ is repelling.

\noindent Let us fix $\varepsilon>0$ and let $u_{max}$ be the point of maximum of the critical manifold $C^1_0$ obtained from the extremum condition $\dot{q_0}(u)=0,$  where $q_0(u)$ is given in Appendix A. 
Solving for $u_{max}$ we find  
\begin{equation} \label{eq:u_max}
u_{max} = \dfrac{(\alpha-\alpha \beta -1)+\sqrt{1+\alpha+\alpha^2-\alpha \beta+\alpha^2\beta+\alpha^2 \beta^2}}{3\alpha}
\end{equation}
and from the expression of $C_0^1$ we have
\begin{equation}\label{eq:v_max}
\begin{aligned}
v_{max} = q(u_{max})\ = \ \gamma (1-u_{max})(u_{max}+\beta)(1+\alpha u_{max}).
\end{aligned}
\end{equation}
Now we consider a trajectory starting from a point, say $(u_{1},v_{1})$, where $u_1 < u_{max}$ and $v_1 = v_{max}$. The trajectory gets attracted toward the attracting manifold $V_+$ and starts moving downward maintaining proximity to $V^+.$ It was expected that the trajectory would leave the vertical axis at the bifurcation point $T$ where it loses its stability \cite{Muratori89}. The trajectory crosses the point $T$ and continues to move vertically downward remaining close to the repelling part $V^-$, for a certain time, until a minimum predator population $p(v_1)$ is attained s.t. $0<p(v_{1})<\gamma \beta$. After leaving the slow manifold near the point $p(v_1)$, the trajectory starts moving along a fast horizontal segment and gets attracted to attracting slow manifold $C_0^{1,a}$. This point of exit is determined by an implicit function $p(v_1)$, called entry-exit function, which is defined implicitly as
\begin{equation*}
\int_{p(v_{1})}^{v_{1}}\dfrac{f_1(0,v)}{vg_1(0,v)}dv=0.
\end{equation*}
For simplicity we define $v_{0}:=p(v_{1})$, then we have 
\begin{equation}\label{eq:en_ex_integral}
\begin{aligned}
\int_{v_{0}}^{v_{1}}\dfrac{v-\gamma\beta}{v\delta}dv = 0\,\,
\implies(v_{1}-v_{0})-\gamma\beta\ln\Big(\frac{v_{1}}{v_{0}}\Big) =0
\end{aligned}
\end{equation}
Substituting (\ref{eq:v_max}) into equation (\ref{eq:en_ex_integral}) we obtain a transcendental equation in $v_0$ which we solve numerically to obtain the exit point.\\ For the parameter values $\alpha=0.5$, $\beta=0.2$, $\gamma=3$, $\delta=0.3$ we obtain $u_{max}=0.472$,\ $v_{1} = 1.316$, and solving the transcendental equation (\ref{eq:en_ex_integral}) we get $p(v_{1}) = 0.207509$, which is the exit point from the manifold $C_0^0$. 
\begin{theorem}
	Let $P$ be the fold point on the critical manifold $C^1_0$ where the slow flow on the attracting manifold $C_{0}^{1,a}$ is given by (\ref{eq:slow-flow}). Also assume that the coexistence equilibrium point lies on the normally hyperbolic repelling critical submanifold under the parametric restriction,
	$$\dfrac{\delta}{1-\alpha \delta}< \dfrac{(\alpha-\alpha \beta -1)+\sqrt{1+\alpha+\alpha^2-\alpha \beta+\alpha^2\beta+\alpha^2 \beta^2}}{3\alpha}$$
	and let $U$ denotes a small neighborhood of a singular trajectory $\gamma_0$ consisting of alternate slow and fast trajectories. Then for sufficiently small $\varepsilon$ there exist a unique attracting limit cycle $\gamma_\varepsilon \subset U$ such that $\gamma_\varepsilon \rightarrow \gamma_0$ as $\varepsilon \rightarrow 0$.
\end{theorem}
\begin{proof}
	The proof is given in Appendix D.
\end{proof} 
\noindent These cycles are shown with the help of a numerical example in Appendix D. 
When $\varepsilon \rightarrow 0$, all the trajectories asymptotically converge to this stable limit cycle consisting of alternate slow and fast transitions of prey and predator densities. This cycle can be interpreted as: when the predator population reaches some high density there is a rapid decline in the prey population due to excessive consumption by the specialist predator and the prey reaches a considerably low level. As a consequence, the predator population declines slowly until it reaches a low threshold density at which the prey population again starts growing. Consequently, the prey regenerates within a very short time while predator density remains more or less fixed. Finally, the predator population starts growing slowly due to the abundance of resources. Finally, when the predator density reaches its maximum level, the slow-fast cycle completes and this dynamics continues with time. 

\subsection{Canard Explosion} \label{sec:Canard_explosion}

In the previous sub-sections, we have observed the periodic dynamics of the slow-fast system near the canard point, where the predator nullcline intersects the non-trivial prey nullcline at the fold point. This occurs at a certain threshold of the parameter $\delta$. At this point, the coexistence equilibrium point loses stability through singular Hopf bifurcation and a small amplitude stable limit cycle is observed. Due to the decrease of the parameter $\delta$, the Hopf bifurcating stable cycle grows in size and settles down to relaxation oscillation. The fast transition in the size of the limit cycle from small canard cycles to relaxation oscillation occurs in an exponentially small range of the parameter $\delta$. This phenomenon is known as the canard explosion.

\noindent The family of canard cycles are already shown in Fig.~\ref{fig:dynamics_slow_fast}(c) for fixed $\varepsilon$ and three values of $\delta$ close to the singular Hopf bifurcation threshold $\delta_H$. The coexistence equilibrium is stable for $\delta>\delta_H$ and the trajectory converges to the stable steady state for any initial condition as it is the global attractor. We can see that for $\delta$ just below $\delta_H$, a stable limit cycle grows in size and a new periodic solution emerges known as the canard cycle without a head  Fig.~\ref{fig:dynamics_slow_fast}(c) (cyan color). This marks the onset of the canard explosion. Further decreasing $\delta$ slightly we obtain another canard cycle known as canard with head Fig.~\ref{fig:dynamics_slow_fast}(c) (blue color). This cycle is special in the sense that from the vicinity of the fold point it follows the repelling slow manifold $C^{1,r}_0$ for $O(1)$ time, before jumping to another attracting manifold. A maximal canard is obtained at $\delta=\delta_c$. After crossing the maximal canard threshold, the system settles down to a large stable periodic solution called relaxation oscillation, which marks the end of canard explosion. This orbit is characterized by the fact that the slow flow on reaching the vicinity of the fold point directly jumps to another attracting slow manifold, as studied in the previous section. The strength of the Allee effect has a significant influence on the amplitude of stable oscillatory coexistence of both the species. The change in the amplitude of the limit cycle corresponding to canard explosion is shown in Fig.~\ref{fig:canard_cycles}.
\begin{figure}[!ht]
	\centering
	\mbox{
		{\includegraphics[width=0.45\textwidth]{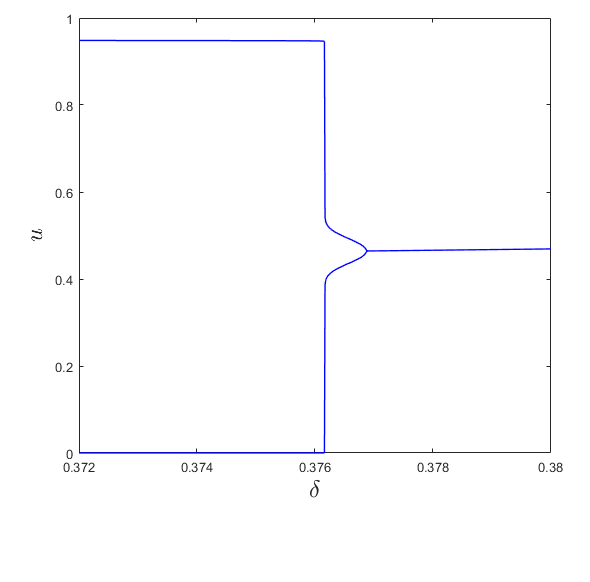}}
	}
	\caption{The bifurcation diagram showing the change in the amplitude of the canard cycles is plotted against $\delta$ for $\alpha=0.5,\ \gamma=3,\ \varepsilon=0.01$ $\beta=0.22$} \label{fig:canard_cycles}
\end{figure}

\noindent For smaller values of $\beta$, the size of the canard cycle is very large and the canard explosion occurs in an exponentially small interval. However, on increasing the value of $\beta$ the size of the limit cycle shrinks and instead of a sudden change in the size of the cycle, we observe a gradual increase in the amplitude of the periodic solution Fig.~\ref{fig:canard_different_allee_strength}. Though the transition from canard cycle to relaxation oscillation takes place in a much wider parametric interval, in this case, it is difficult to distinctively identify the different periodic solutions. For smaller values of $\beta$, when the prey population is almost absent, the predator population also slowly declines to an almost endemic level. But on increasing the strength of the Allee effect the predator density never collapses but stays at the system at a much higher density. After which the system experience a sudden outbreak in the prey population within a very short interval of time. Again because of the abundance of resources, the predator population grows slowly and reaches the maximum capacity. Due to the exploitation of the resources, there is a fast decline in the prey density and this cycle continues.
\begin{figure}[!ht]
	\centering
	\mbox{\subfigure[]{\includegraphics[width=0.45\textwidth]{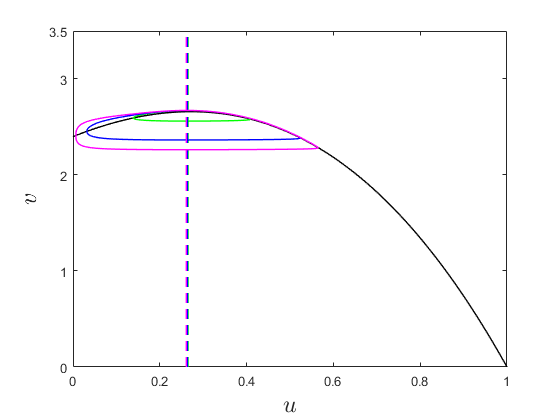}}}
	\mbox{\subfigure[]{\includegraphics[width=0.45\textwidth]{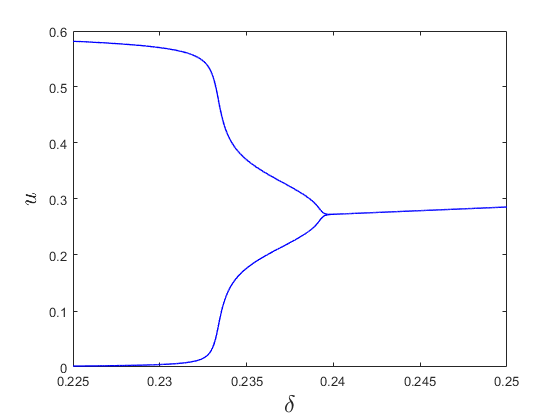}}}
	\caption{(a) The canard cycles for $\alpha=0.5,\ \beta=0.8,\ \gamma=3,\ \varepsilon=0.01$ and for different values of $\delta$ i.e $ \delta=0.234$ (green), $\delta=0.233$ (blue), $\delta=0.231$ (magenta), (b) the bifurcation diagram showing the change in the size of the cycles.} \label{fig:canard_different_allee_strength}
\end{figure}

\section{Spatio-temporal model}\label{sec:spatial}

We now consider the spatio-temporal model corresponding to the model (\ref{eq:temp_weak_fast}) with slow-fast time scale. Here, we consider that the prey and predator densities are functions of time and space, $u(t,{\bf x})$ and $v(t,{\bf x})$ denote prey and predator densities, respectively, at time $t$ and at spatial location ${\bf x}$. In case of one dimensional (1D) space ${\bf x}=x\in\mathbb{R}$ and for two dimensional (2D) space ${\bf x}=(x,y)\in\mathbb{R}^2$. For simplicity we assume that ${\bf x}$ belongs to a bounded domain $D\subset\mathbb{R}$ and $D\subset\mathbb{R}^2$ respectively. The spatio-temporal dynamics of the prey-predator interaction is described by the following reaction-diffusion equation
\begin{subequations} \label{eq:sptemp_weak}
	\begin{eqnarray}
	u_t &=& \gamma u(1-u)(u+\beta)-\frac{uv}{1+\alpha u}+\nabla^2u,\\
	v_t &=& \varepsilon \left(\frac{uv}{1+\alpha u}-\delta v\right)+d\nabla^2v,
	\end{eqnarray}
\end{subequations}
where $d$ is the ratio of diffusivity coefficients of predator to prey and $\nabla^2$ is the Laplacian operator. The above spatio-temporal model is subject to no-flux boundary condition and non-negative initial condition. The model (\ref{eq:sptemp_weak}) can not produce any stationary Turing pattern and it can be proved that the Turing instability condition is not satisfied. However, instability of the coexistence steady-state due to Hopf bifurcation combined with the diffusivity of two species leads to some dynamic pattern due to the formation of traveling wave, wave of invasion and spatio-temporal chaos. The mechanisms responsible for such kind of pattern formation are described in \cite{Lewis16}.

\noindent Analytical condition for the existence of traveling wave leads to successful invasion by specialist predator is derived in the next subsection. For simplicity of mathematical calculation, we restrict ourselves to one dimensional space to explain the existence of traveling wave. In case of two dimensional spatial domain, the analogous patterns are presented separately.

\subsection{Existence of Traveling wave}

To study the successful invasion by the predator we consider the system (\ref{eq:sptemp_weak}) in one dimensional space, we re-write the above system as
\begin{subequations} \label{eq:sptem_onedim}
	\begin{eqnarray}
	\frac{\partial u(t,x)}{\partial t} &=& f(u,v) +\frac{\partial^2u}{\partial x^2},\\
	\frac{\partial v(t,x)}{\partial t} &=& \varepsilon g(u,v) +d\frac{\partial^2v}{\partial x^2},
	\end{eqnarray}
\end{subequations}
where $f(u,v) = \gamma u(1-u)(u+\beta)-\frac{uv}{1+\alpha u},\ g(u,v)= \frac{uv}{1+\alpha u}-\delta v.$ The predator is introduced in a small domain where the prey density is at its carrying capacity. The successful invasion of the predator is characterized by the existence of a traveling wave joining the predator free steady-state with the coexistence steady-state. Depending upon the  stability of the coexistence steady-state, we can find monotone traveling wave, non-monotonic traveling wave, and periodic traveling wave as explained below with the help of numerical examples. 

\noindent We begin with deriving the minimum speed of the traveling wave which result in the successful invasion of the specialist predator into the space already inhabited by its prey. For this, we first consider a single-species model with the linear growth: 
$$\frac{\partial v(t,x)}{\partial t} = \alpha v+D\frac{\partial^2v}{\partial x^2},$$
where $\alpha$, $D\,>\,0$ are parameters with obvious meaning. Stricktly speaking, the above equation does not possess a traveling wave solution. However, for a compact initial condition, it is known that the tail of the profile propagates with the constant speed given by $c_{min}=2\sqrt{D\alpha}$, see \cite{Lewis16}, sometimes referred to as the Fisher spreading speed. 

\noindent For the invasion of predator into space inhabited by its prey, we consider the tail of the profile where $u\approx 1$ and $v\approx 0$ and linearize (\ref{eq:sptem_onedim}b) around $(1,0)$: 
\begin{equation} \label{eq:TW_linearized}
\frac{\partial v(t,x)}{\partial t} = \varepsilon \Big(\frac{1}{\alpha+1}-\delta \Big)v+d\frac{\partial^2v}{\partial x^2}. 
\end{equation}
Clearly, the speed of the traveling wave, at the onset of successful invasion, is given by
\begin{equation}
c_v = 2\Big(\varepsilon d \Big[\frac{1}{\alpha+1}-\delta\Big] \Big)^{1/2}.
\end{equation}
The feasibility condition is $\delta (\alpha + 1)<1$. The expression for $c_v$ indicates that the speed of traveling wave reduces in the order of $\sqrt{\varepsilon}$. We consider the existence of traveling wave starting from the predator free steady-state, which leads to the successful establishment of the predators, this requires the consideration of the system (\ref{eq:sptem_onedim}) with the following conditions
$$ u(t,x) = 1,\ \text{and \ } v(t,x) = 0,\ \text{as\ } x\rightarrow -\infty,\ \forall \,t, $$
$$ u(t,x) = u_*,\ \text{and \ } v(t,x) = v_*,\ \text{as\ } x\rightarrow \infty,\ \forall \,t .$$
We consider the traveling wave solution of the system (\ref{eq:sptem_onedim}) in the form $u(t,x)=\phi(\xi)$, $v(t,x)=\psi(\xi)$ where $\xi = x-ct$ and $c$ is the wave speed. The functions $\phi(\xi)$ and $\psi(\xi)$ thus satisfy the equations
\begin{equation} \label{eq:TW_2ODE}
\begin{aligned}
\dfrac{d^2\phi}{d\xi^2} + c\dfrac{d\phi}{d\xi} + f(\phi,\psi) &= 0,\\
d\dfrac{d^2\psi}{d\xi^2} + c\dfrac{d\psi}{d\xi} + \varepsilon g(\phi,\psi)  &=0.
\end{aligned} 
\end{equation}
Substituting $p(\xi)= -\dfrac{d\phi}{d\xi}$ and $q(\xi)= -\dfrac{d\psi}{d\xi}$, from (\ref{eq:TW_2ODE}) we can derive four coupled ordinary differential equations as follows
\begin{equation} \label{eq:TW_4ODE}
\begin{aligned}
\dfrac{d\phi}{d\xi} &= -p,\\
\dfrac{dp}{d\xi} & = -cp + f(\phi,\psi),\\
\dfrac{d\psi}{d\xi} &= -q,\\
\dfrac{dq}{d\xi} &= \dfrac{1}{d}(-cq+\varepsilon g(\phi,\psi)).
\end{aligned}
\end{equation}
Three homogeneous steady states ($E_0,\ E_1,\ E_*$) of the spatio-temporal model (\ref{eq:sptem_onedim})  corresponds to three steady states of system (\ref{eq:TW_4ODE}) are $Q_0(0,0,0,0)$, $Q_1(1,0,0,0)$ and $Q_*(u_*,0,v_*,0)$.
To ensure the successful invasion of the predator, we focus on the dynamics of the system (\ref{eq:TW_4ODE}) around $Q_1$ and $Q_*$. The Jacobian matrix of the system (\ref{eq:TW_4ODE}) evaluated at $Q_1$ is
\begin{equation}
J_{Q_1} = 
\begin{pmatrix}
0&-1&0&0\\ -\gamma(1+\beta)& -c& -\frac{1}{1+\alpha}&0\\0&0&0&1\\0&0& \frac{\varepsilon}{d}(\frac{1}{1+\alpha}-\delta)&-\frac{c}{d}
\end{pmatrix}  
\end{equation}
The eigenvalues of the matrix $J_{Q_1}$ are $\lambda _{1,2} = \frac{c}{2}\pm \frac{\sqrt{c^2+4g(1+\beta)}}{2}$ and 
$\lambda _{3,4} = -\frac{c(1+\alpha)}{2}\pm \frac{\sqrt{\Gamma}}{d(1+\alpha)}$ where $\Gamma=(1+\alpha)^2c^2-4\varepsilon d(1+\alpha)(1-\delta-\alpha\delta)$. First two eigenvalues are real, whereas $\lambda_{3,4}$ are real for $c^2\ge \frac{4\varepsilon d(1-\delta-\alpha\delta)}{1+\alpha}$. The traveling wave exist if all the eigenvalues are real, otherwise the trajectories will spiral around $Q_1$ and leads to negative population density. Hence, the minimum speed of traveling wave originating from predator free steady-state is 
\begin{equation}\label{eq:c_min}
c_{min} = \Big[\frac{4\varepsilon d(1-\delta-\alpha\delta)}{1+\alpha}\Big]^{1/2}.
\end{equation}
Note that $c_{min}$ depends on $\varepsilon$, thus for $\varepsilon<1$, the wave speed decreases. The minimum wave speed derived here is the same as it was derived earlier with the help of the linearized equation.
The expressions for the eigenvalues of the Jacobian matrix $J_{Q_*}$ are quite complicated and hence we avoid writing them here explicitly for the sake of brevity. For $c\ge c_{min}$, the eigenvalues of $J_{Q_*}$ can be real or complex depending on which we find monotone traveling wave and non-monotone as well as periodic traveling wave originating from the steady states $E_1$.

\begin{figure}[!ht]
	\begin{center}
		\mbox{\subfigure[$t=300$]{\includegraphics[width=0.45\textwidth]{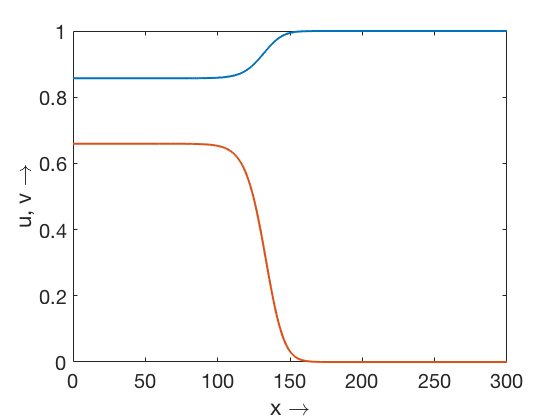}}
			\subfigure[$t=240$]{\includegraphics[width=0.45\textwidth]{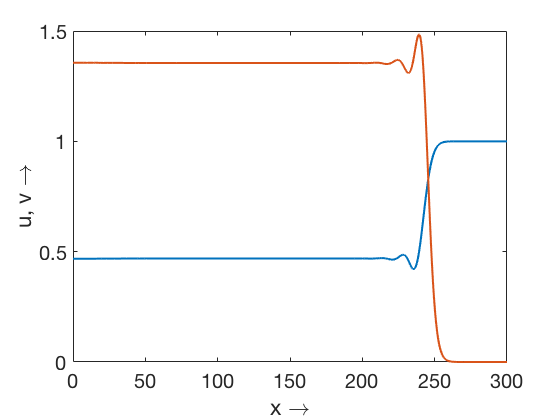}}}
		
		\mbox{\subfigure[$t=220$]{\includegraphics[width=0.45\textwidth]{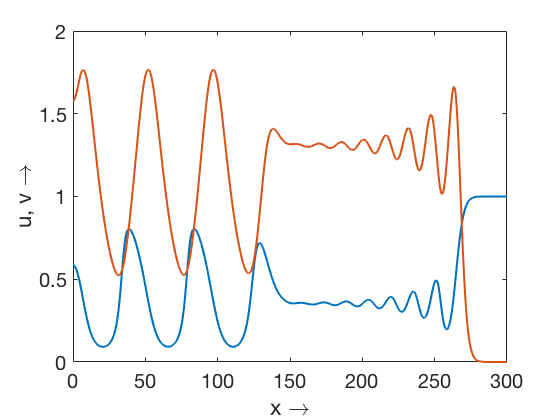}}
		}	
		\caption{Monotone traveling wave, non-monotone traveling wave and periodic traveling wave obtained for the model (\ref{eq:sptem_onedim}) with $\alpha = 0.5$, $\beta = 0.22$, $\gamma = 3$, $\varepsilon = 1$ and (a) $\delta=0.6$, (b) $\delta=0.38$ and (c) $\delta=0.3$ at different time as mentioned with the figures.}\label{fig:TW_1d_1}
	\end{center}
\end{figure}

To illustrate the existence of various type of traveling waves, we fix the parameter values $\alpha=0.5$, $\beta=0.22$,  $\gamma=3$, $\varepsilon=1$ and consider $\delta$ as variable parameter. From (\ref{eq:c_min}) we find that traveling wave exists for $\delta<2/3$. The eigenvalues of $J_{Q_*}$ are real for $0.52<\delta<0.667$ and the eigenvalues are complex conjugate for $\delta\le 0.52$. Complex conjugate eigenvalues with negative real parts correspond to non-monotone traveling wave and with positive real part correspond to periodic traveling wave. Three types of traveling waves are shown in Fig.~\ref{fig:TW_1d_1} for three different values of $\delta$. For $\delta\, =\,0.38$, the minimum wave speed $c_{min}\approx 1.07$, and the complex conjugate eigenvalues with negative real part are $-1.22\pm0.423i$, whereas for $\delta=0.3,\ c_{min}\approx 1.211$ and the complex eigenvalues with positive real part are $0.034\pm 0.405i$. The initial condition used in Fig.~\ref{fig:TW_1d_1} is given by
$$u(0,x)\,=\,\left\{\begin{array}{ll}
u_*, & 0\leq x\leq3\\
1, & 3< x\leq 300 \\
\end{array}\right.,\,\,\,
v(0,x)\,=\,\left\{\begin{array}{ll}
v_*, & 0\leq x\leq3\\
0, & 3< x\leq 300 \\
\end{array}\right..$$

\noindent The traveling wave emerging from the above initial conditions connects the prey-only steady state $E_1$ to the coexistence state $E_*$. Its shape depends on parameter values; for $\delta=0.6$, its profile is monotone (see Fig.~\ref{fig:TW_1d_1}a). A non-monotone traveling wave emerges for $\delta=0.38$ (see Fig.~\ref{fig:TW_1d_1}b). The range of spatio-temporal oscillation increases for values of $\delta$ close to the temporal Hopf bifurcation threshold. For $\delta=0.3$, we find periodic traveling wave, with a plateau behind the oscillatory front corresponding to steady state $E_*$ which is, for these parameter values, unstable: a phenomenon known as the dynamical stabilization \cite{Malchow02,Petrovskii00,Petrovskii01TPB,Sherratt98}.

\begin{figure}[ht!]
	\begin{center}
		\mbox{\subfigure[]{\includegraphics[width=0.45\textwidth]{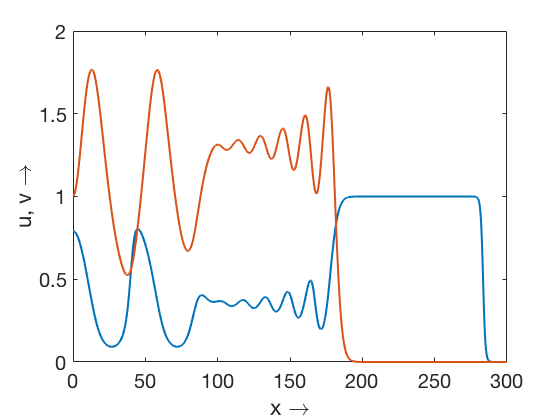}}
			\subfigure[]{\includegraphics[width=0.45\textwidth]{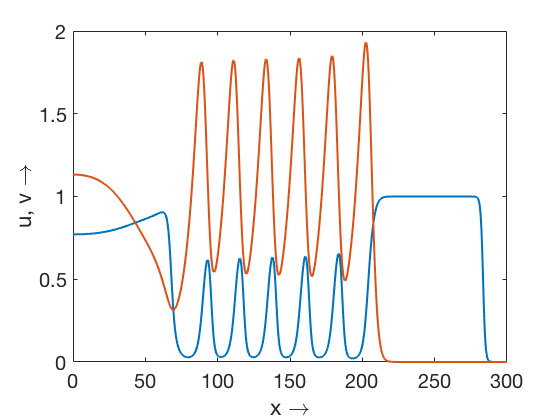}}
		}
		\mbox{\subfigure[]{\includegraphics[width=0.45\textwidth]{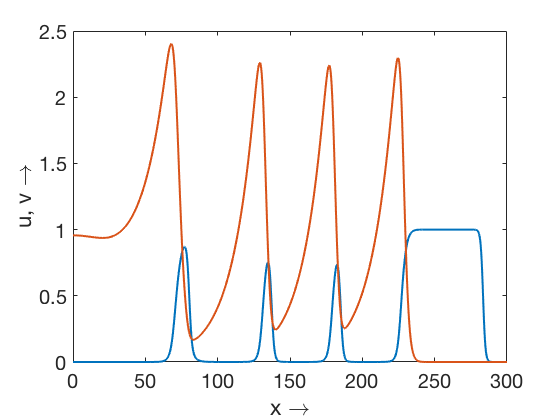}}
		}	
		\caption{Periodic traveling waves for the model (\ref{eq:sptem_onedim}) with $\alpha = 0.5$, $\beta = 0.22$, $\gamma = 3$, $\varepsilon = 1$ and (a) $\delta=0.3$, (b) $\delta=0.2$ and (c) $\delta=0.1$ at the same instant of time $t=160$.}\label{fig:PTW_1d}
	\end{center}
\end{figure}

We mention here that the properties of emerging traveling wave are rather robust with regard to the choice of initial conditions. 
For instance, for a different initial condition as given by
$$u(0,x)\,=\,\left\{\begin{array}{ll}
1, & 0\leq x\leq3\\
0, & 3< x\leq 300 \\
\end{array}\right.,\,\,\,
v(0,x)\,=\,\left\{\begin{array}{ll}
0.2, & 0\leq x\leq2\\
0, & 2< x\leq 300 \\
\end{array}\right.,$$
the emerging periodic traveling waves shown in Fig.~\ref{fig:TW_1d_1}c and Fig.~\ref{fig:PTW_1d}a are qualitatively similar.
Numerical simulation for smaller values of $\delta$ shows an increase in the magnitude and period of the oscillating front as shown in Fig.~\ref{fig:PTW_1d}b-c. For numerical simulations we have chosen a spatial domain of size $[0,300]$, further increase in domain size does not effect the qualitative property of the traveling waves.

To understand the effect of slow-fast time scale on the resulting pattern formation, here we consider the change in periodic traveling wave shown in Fig.~\ref{fig:PTW_1d}a for $\varepsilon<1$. The model (\ref{eq:sptem_onedim}) is simulated for three different values of $\varepsilon$ as mentioned at the caption of Fig.~\ref{fig:PTW_SF_1d}. With the decrease in $\varepsilon$, we observe that the oscillatory wake of the invading predator front separating the predator-free area and the onset of spatiotemporal shrinks and eventually disappears for smaller values of $\varepsilon$, so that dynamical stabilization does not occur. Also the size of predator-free area where prey exists at its carrying capacity increases with the decrease in $\varepsilon$. The size of predator-free patch increases further for $\varepsilon<0.25$.
\begin{figure}[!ht]
	\begin{center}
		\mbox{\subfigure[]{\includegraphics[width=0.45\textwidth]{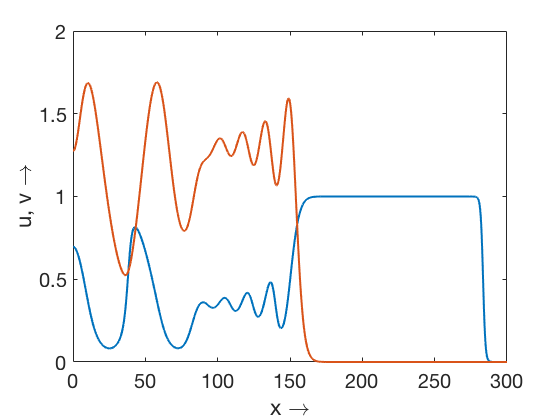}}
			\subfigure[]{\includegraphics[width=0.45\textwidth]{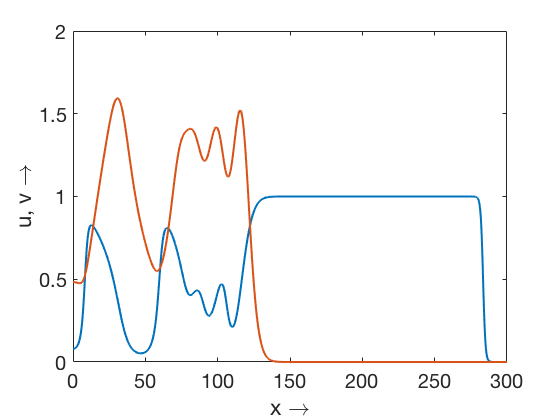}}}
		
		\mbox{\subfigure[]{\includegraphics[width=0.45\textwidth]{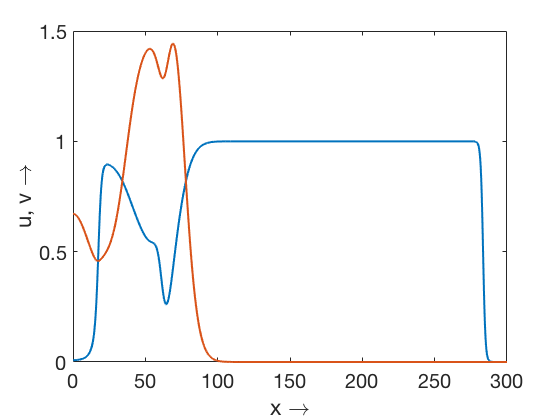}}
		}	
		\caption{Periodic traveling waves for the model (\ref{eq:sptem_onedim}) with $\alpha = 0.5$, $\beta = 0.22$, $\gamma = 3$, $\delta=0.3$ and (a) $\varepsilon = 0.75$, (b) $\varepsilon = 0.5$ and (c) $\varepsilon = 0.25$ at the same instant of time $t=160$.}\label{fig:PTW_SF_1d}
	\end{center}
\end{figure}

\subsection{Patterns in two dimension}

Finally, we consider the spatio-temporal pattern formation over two dimensional spatial domain. The nonlinear reaction-diffusion system (\ref{eq:sptemp_weak}) is solved numerically using five-point finite difference scheme for the Laplacian operator and forward Euler scheme for the temporal part with the initial conditions (\ref{eq:IC1_1}) and (\ref{eq:IC2}). Equal diffusivities are considered throughout {\it i.e}, $d=1$. Two types of initial conditions are used to study the successful invasion and establishment of both species. In \cite{Morozov06} Morozov {\it et.al} used the following initial condition which explains that a small amount of population of both the species are introduced within a small elliptic domain
\begin{equation}\label{eq:IC1_1}
u(0,x,y)\,=\,\left\{
\begin{array}{ll}
u_0, & \frac{(x-x_1)^2}{\Delta_{11}}+\frac{(y-y_1)^2}{\Delta_{12}}\,\le\,1\\
0,  & \text{otherwise}\\
\end{array}\right. \text{,}\ \
v(0,x,y)\,=\,\left\{
\begin{array}{ll}
v_0, & \frac{(x-x_2)^2}{\Delta_{21}}+\frac{(y-y_2)^2}{\Delta_{22}}\,\le\,1\\
0,  & \text{otherwise}\\
\end{array}\right.
\end{equation}
where $u_0$ and $v_0$ measure the initial densities of prey (native) and predator (invasive) species respectively. The other initial condition we consider here, is a small amplitude heterogeneous perturbation from the homogeneous steady states (see \cite{Medvinsky02} for details), is 
\begin{equation}\label{eq:IC2}
\begin{aligned}
u(x,y,0) &=& &u_* - e_1(x-0.1y-225)(x-0.1y-675),&\\
v(x,y,0) &=& &v_* - e_2(x-450)-e_3(y-450),&
\end{aligned}
\end{equation}
where $(u_*,v_*)$ is the homogeneous steady state and for numerical simulation we choose $e_1 = 2\times 10^{-7}$, $e_2 = 3\times 10^{-5}$, and $e_3 = 2\times 10^{-4}$.

\noindent First we simulate the spatio-temporal model with the initial condition (\ref{eq:IC1_1}) in a square domain $L\times L$ with $L=300$, with grid spacing $\Delta x = \Delta y=1$ and time step $\Delta t =0.01$. The simulation results are verified with other choices of $\Delta x$ and $\Delta t$ to ensure that the obtained results are free from numerical artifact. We consider a small elliptic domain within which the prey is at its carrying capacity ($u_0=1$) and a small amount of population ($v_0=0.2$) is introduced. Other parameter values are $x_1 = 153.5$, $y_1 = 145$, $x_2 = 150$, $y_2=150$, $\Delta_{11} = 12.5$, $\Delta_{12} = 12.5$, $\Delta_{21} = 5$, $\Delta_{22}=10$. We simulate the model for a sufficiently long time so that the invading waves can cover the whole domain and hits the domain boundary. Parameter values of $\alpha$, $\beta$, $\gamma$ and $\varepsilon$ are same as mentioned in the previous subsection. We will first check the change in resulting pattern by varying the parameter $\delta$. In Fig \ref{fig:weak_allee_delta_0dot3}, two dimensional spatial distribution of prey population density at two different time points are shown for $\delta=0.3$. We have omitted the spatial distribution of the predator population as they exhibit similar patterns as exhibited by the prey species. Choosing $\delta=0.3$, just below the temporal Hopf bifurcation threshold ($\delta_H= 0.3768$), we observe concentric circular rings as the initial transient pattern which eventually settle down to an interacting spiral pattern once the transients are over. With the advancement of time, the expanding circular rings hit the domain boundary and break into irregular patches. These irregular spiral patches cover the whole domain, and the system dynamics can be identified as interacting spiral chaos (see Fig.~\ref{fig:weak_allee_delta_0dot3}b). Note that the initial invading waves are the periodic traveling waves (Fig.~\ref{fig:weak_allee_delta_0dot3}a) but in large-time we observe irregular spatio-temporal oscillations (Fig.~\ref{fig:weak_allee_delta_0dot3}b). This type of chaotic dynamics persists in the vicinity of temporal Hopf bifurcation threshold ($0.2<\delta<\ \delta_H=0.3768$).

\begin{figure}[ht!]
	\begin{center}
		\mbox{\subfigure[$t=80$]{\includegraphics[width=0.45\textwidth]{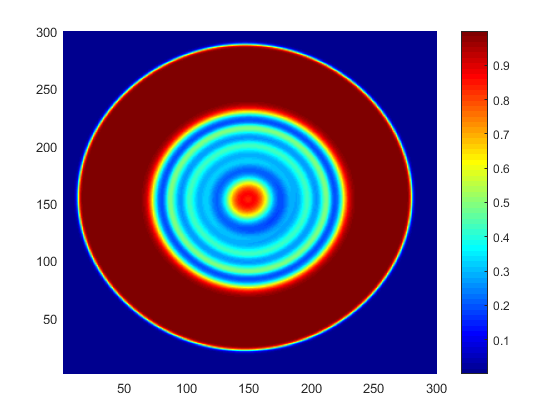}}
			\subfigure[$t=500$]{\includegraphics[width=0.45\textwidth]{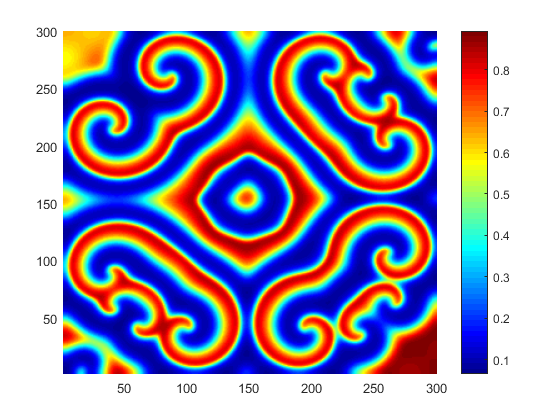}}
		}	
		\caption{Spatio-temporal pattern of prey with initial condition \ref{eq:IC1_1} obtained for $\alpha = 0.5,\ \beta = 0.22,\ \gamma = 3,\ \delta=0.3,\ \varepsilon = 1$\ at different intervals of time. }\label{fig:weak_allee_delta_0dot3}
	\end{center}
\end{figure}

\noindent Keeping other parameters fixed, we further decrease $\delta \  (\le 0.2)$ and find propagating circular rings which are periodic traveling waves. The number of rings, that is the number of population patches within the fixed domain, decreases with the decrease in magnitude of $\delta$. These periodic traveling waves do not break after hitting the boundary and the spatio-temporal dynamics remains unaltered. This result is in agreement with Sherrat et. al \cite{Sherratt97} that the system exhibits oscillatory dynamics as a successful invasion. Periodic traveling fronts for $\delta\le 0.2.$ are shown in Fig.~\ref{fig:weak_allee_delta_0dot1}. 

\begin{figure}[ht!]
	\begin{center}
		\mbox{\subfigure[$\delta=0.2$]{\includegraphics[width=0.45\textwidth]{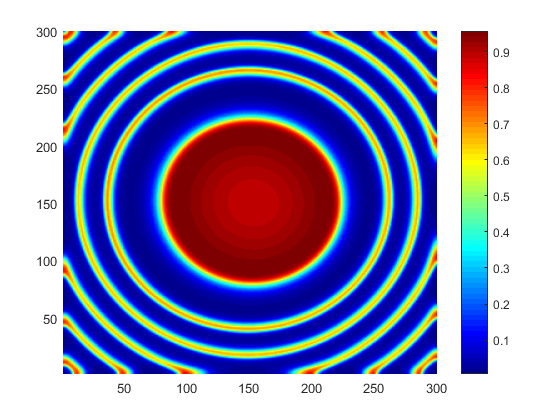}}
			\subfigure[$\delta=0.1$]{\includegraphics[width=0.45\textwidth]{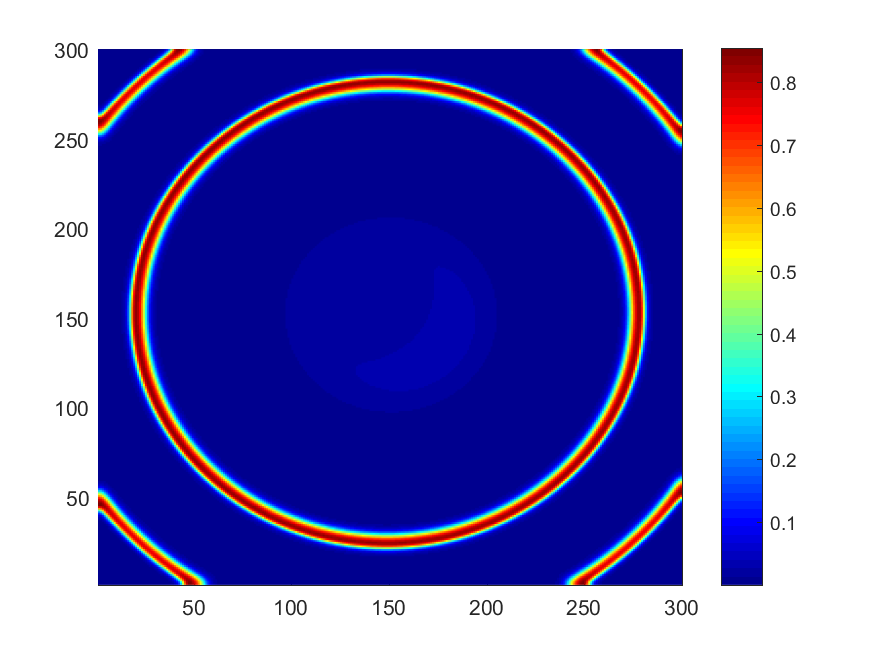}}}
		
		\mbox{\subfigure[$\delta=0.05$]{\includegraphics[width=0.45\textwidth]{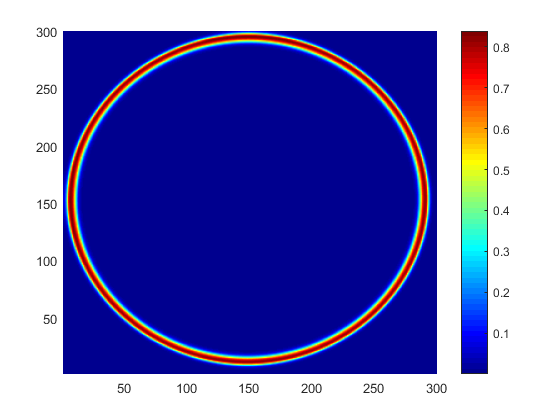}}}
		\caption{Spatio-temporal pattern of prey with initial condition \ref{eq:IC1_1} for $\alpha = 0.5,\ \beta = 0.22,\ \gamma = 3,\ \varepsilon = 1$\ at $t=200$.
		}\label{fig:weak_allee_delta_0dot1}
	\end{center}
\end{figure}

\noindent Now we consider the effect of the slow-fast time scale on the resulting patterns. In the previous subsection, we have explained the reduction of traveling wave speed with the decrease in the magnitude of $\varepsilon$. As a result, the time taken by the predators to invade over the entire domain for $\varepsilon\ll1$ is much longer as compared to $\varepsilon=1$. For $\varepsilon<1$ we find two kinds of distinctive changes in the resulting patterns. The width of the population patches increase (see Fig~\ref{fig:weak_allee_eps2}a), and the spatio-temporal chaotic dynamics changes to periodic temporal oscillation of nearly homogeneous distribution of prey and predator densities (see Fig~\ref{fig:weak_allee_eps2}d). The time evolution of the spatial average of the prey and predator population is analogous to the temporal Canard cycle.

\begin{figure}[ht!]
	\begin{center}
		\mbox{\subfigure[]{\includegraphics[width=0.4\textwidth]{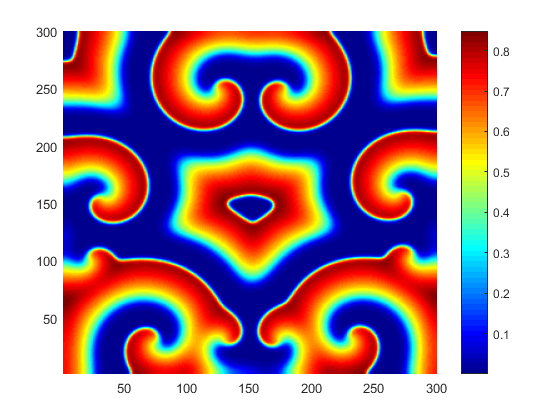}}
			\subfigure[]{\includegraphics[width=0.4\textwidth]{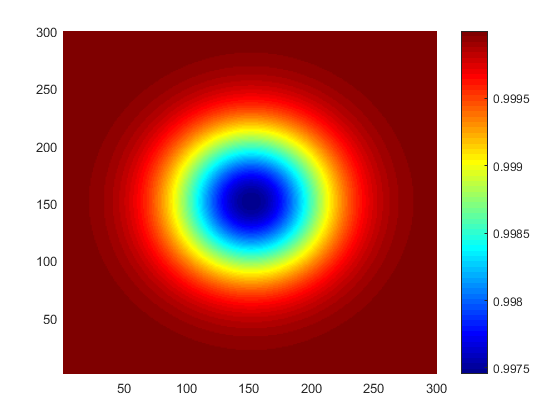}}
		}
		\mbox{\subfigure[]{\includegraphics[width=0.4\textwidth]{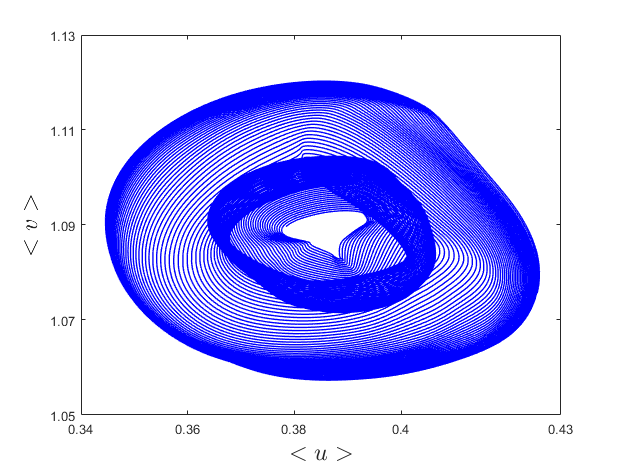}}
			\subfigure[]{\includegraphics[width=0.4\textwidth]{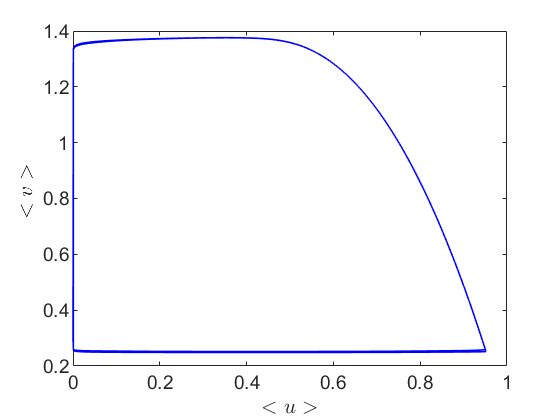}}
			
		}\caption{Spatio-temporal pattern of prey and plots of average spatial of prey and predator population is given w.r.t time for $\alpha = 0.5,\ \beta = 0.22,\ \gamma = 3,\ \delta = 0.3$ for $\varepsilon=0.1$(left) and $\varepsilon=0.01$(right). Upper panel shows the pattern at (a) $t=10000$, (b) $t=5000$; lower panel shows the phase trajectory of spatially averaged densities (c) $t \in [2000, 10000]$, (d) $t \in [5000,10000].$ }\label{fig:weak_allee_eps2}
	\end{center}
\end{figure}
\noindent The choice of the initial condition and the domain size plays an important role in pattern formation. We simulate the system (\ref{eq:sptemp_weak}) with the second kind of initial condition (\ref{eq:IC2}) and over a square domain $L\times L$ with $L=900$. The initial condition indicates a small heterogeneous perturbation from the homogeneous steady state $(u_*,v_*)$. Choosing the same parameter set as Fig.~\ref{fig:weak_allee_delta_0dot3}, initially, we find two spirals rotating about their fixed centers. The regular spirals are destroyed with the advancement of time and the interacting spiral pattern engulfs the whole domain (see Fig.~\ref{fig:IC2_delta0.3}b-c). These patches move, break and form new patches, but qualitatively, the dynamics of the system does not alter with time. 

\begin{figure}[ht!] 
	\centering
	\mbox{\subfigure[$t=360$]{\includegraphics[width=0.45\textwidth]{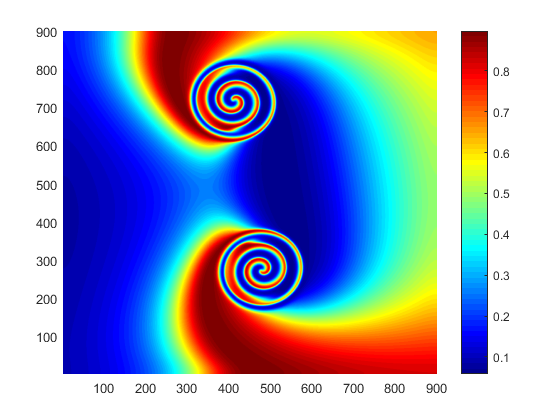}}
		\subfigure[$t=900$]{\includegraphics[width=0.45\textwidth]{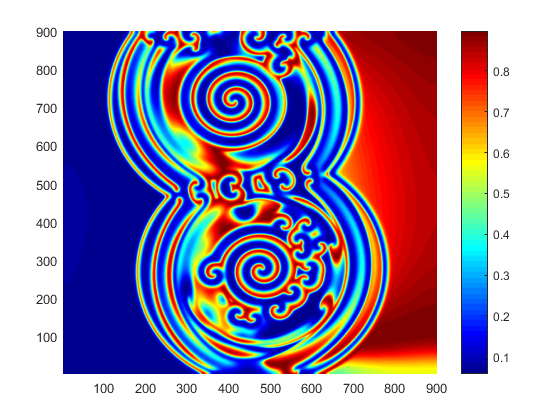}}}
	\mbox{\subfigure[$t=1500$]{\includegraphics[width=0.45\textwidth]{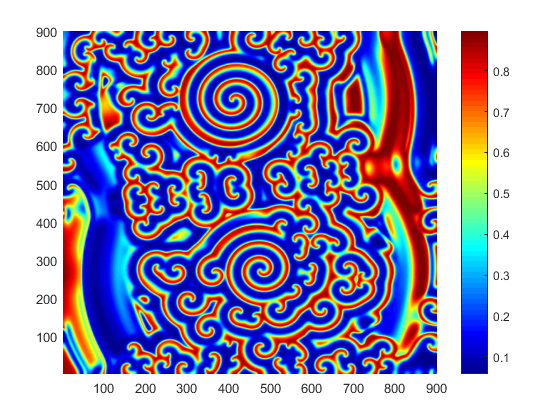}}}
	\caption{Spatio-temporal pattern of prey density for $\alpha = 0.5,\ \beta = 0.22,\ \gamma = 3,\ \delta = 0.3, \ \varepsilon = 1,$ at different time intervals.  }\label{fig:IC2_delta0.3}
\end{figure} 
\noindent Exhaustive numerical simulation indicates that for $\delta$ close to the temporal Hopf bifurcation threshold, the system always exhibits spatio-temporal chaos. The duration and type of transient patterns depend upon the initial condition and the size of the domain. Now considering $\varepsilon < 1$, we found the persistent interacting spirals with thick arms (see Fig.~\ref{fig:IC2_delta0.3_eps0dot1}a for $\varepsilon=0.1$). The regular spiral grows in size and doesn't breakdown even after hitting the boundary when $\varepsilon$ is significantly small, say $\varepsilon=0.01$. The irregularity of the temporal evolution of spatial averages of both the population decreases and moves towards periodic or quasi-periodic oscillation with the decrease in the magnitude of $\varepsilon$. This claim is justified from the time evolution of spatial averages as presented in the lower panel of Fig.~\ref{fig:IC2_delta0.3_eps0dot1}.

\begin{figure}[ht!] 
	\centering
	\mbox{\subfigure[$\varepsilon=1$]{\includegraphics[width=0.3\textwidth]{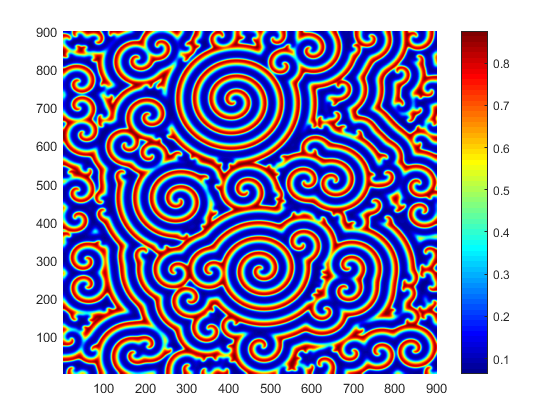}}
		\subfigure[$\varepsilon=0.1$]{\includegraphics[width=0.3\textwidth]{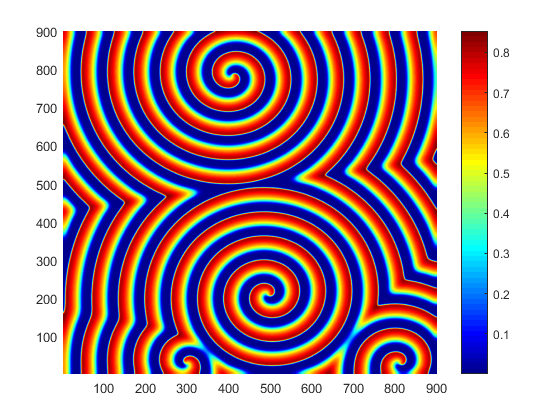}}		
		\subfigure[$\varepsilon=0.01$]{\includegraphics[width=0.3\textwidth]{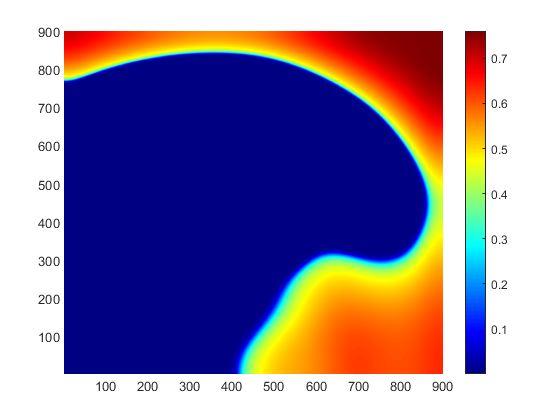}}}
	\mbox{\subfigure[$\varepsilon=1$]{\includegraphics[height = 3.2cm, width=0.3\textwidth]{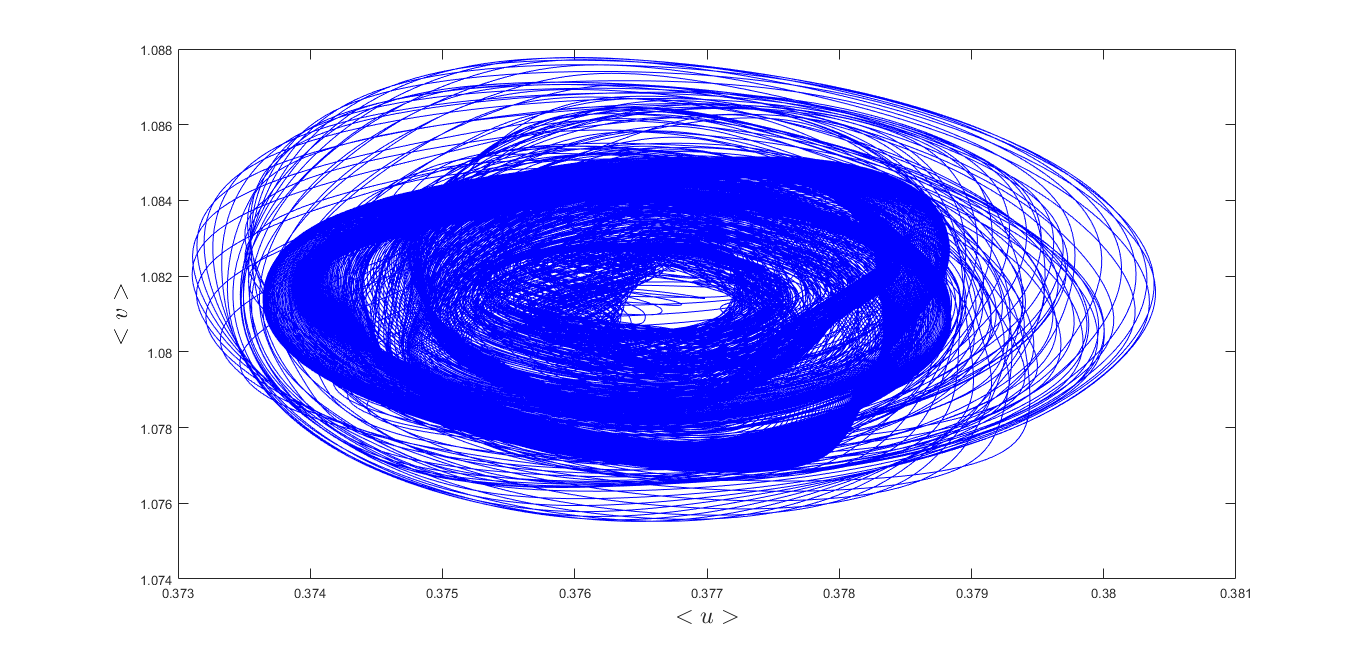}}
		\subfigure[$\varepsilon=0.1$]{\includegraphics[height = 3.2cm, width=0.3\textwidth]{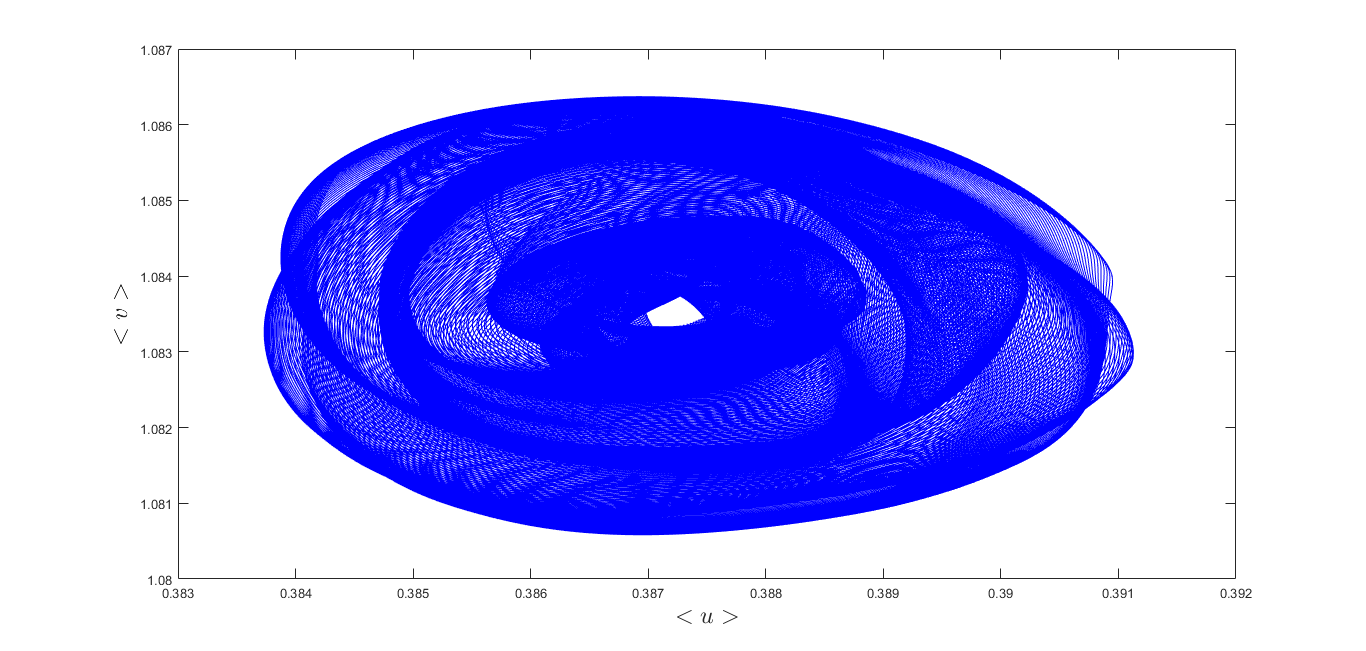}}		
		\subfigure[$\varepsilon=0.01$]{\includegraphics[height = 3.2cm,width=0.3\textwidth]{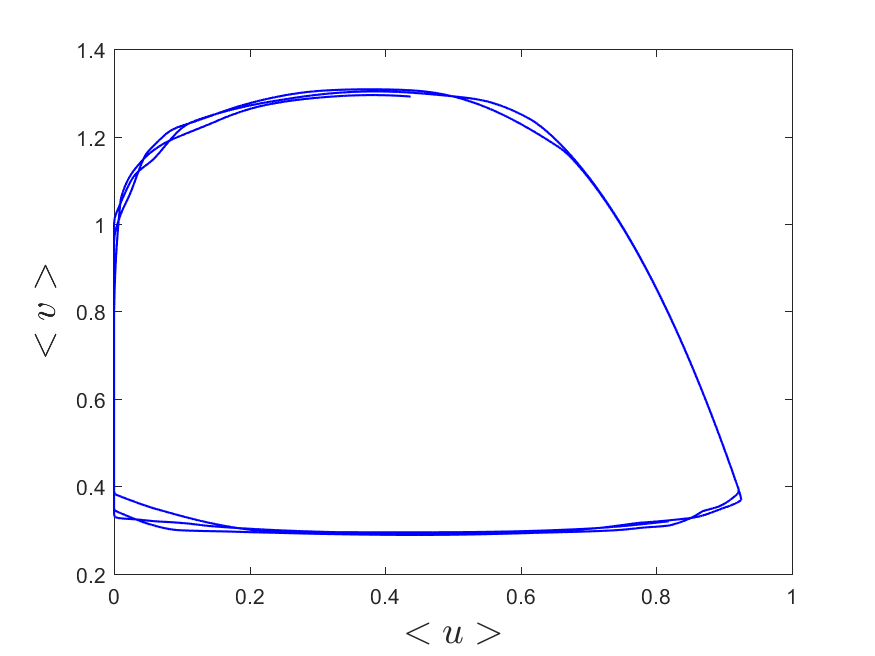}}}
	\caption{Upper panel (a,b,c) represents spatial distribution of prey for $\alpha = 0.5,\ \beta = 0.22,\ \gamma = 3,\ \delta = 0.3$ 
		for different values of $\varepsilon$ (a) $t=12000$, (b) $t=55000$, (c) $t=10000$ ; lower panel (d,e,f) represents phase trajectory of spatially averaged densities obtained after removing initial transients (d) $t\in [3000,12000],$ (e) $t\in [3000,55000]$, (f) $t\in [5000,10000]$.}\label{fig:IC2_delta0.3_eps0dot1}
\end{figure}

\section{Discussion and Conclusions}

Understanding the effects that the existence of multiple time scales may have on the population dynamics of corresponding interacting species, in particular by promoting or hampering their persistence, has been attracting an increasing attention over the last two decades. In particular, some preliminary yet significant work has been done to understand changes in the oscillatory coexistence in the presence of slow-fast time scales \cite{Hek10,Kooi18,Muratori89,Rinaldi92}. However, the effects of the slow-fast dynamics in the spatially explicit systems, e.g.~as given by the corresponding reaction–diffusion equations, remains poorly investigated. This paper aims to bridge this gap, at least partially. As our baseline system, we consider the classical Rosenzweig-McCarthur prey-predator model with the multiplicative weak Allee effect in prey's growth. We pay particular attention to the interplay between the strength of the weak Allee effect (quantified by parameter $0<\beta<1$) and the difference in the time scales for prey and predator (quantified by $\epsilon\le 1$). 

We first provide a detailed slow-fast analysis for the corresponding nonspatial system. In doing that, we have obtained the following results:

\begin{itemize}
	
	\item in the presence of slow-fast dynamics ($\epsilon\ll 1$) and a weak Allee effect, a decrease in the predator mortality may lead to a regime shift where small-amplitudfe oscillations in the populaton abundance change to large-amplitude oscillations (see Fig.~\ref{fig:canard_cycles}). This change becomes more abrupt in case the Allee effect is `not too weak' (i.e.~$\beta$ is sufficiently small), cf.~Figs.~\ref{fig:canard_cycles} and \ref{fig:canard_different_allee_strength}.
	
\end{itemize}

On a more technical side, 
we have derived an asymptotic expansion in $\varepsilon$ for the invariant approximated manifolds and have explained the dynamics of the system near the hyperbolic submanifolds. This theory cannot be extended at non-hyperbolic points. To unravel the complete geometry of the manifolds and their intersection as they pass through the non-hyperbolic points we followed the blow-up technique \cite{Dumortier96,Krupa01A,Krupa01B}. We considered the slow-fast normal form of the model by translating the fold point to the origin. As the transformed system has a singularity at the origin, it is then blown up to a sphere $S^3$ and the trajectories of the blow-up system are mapped on and around the sphere. Using the blow-up analysis we have found the analytical expression for the singular Hopf bifurcation curve $(\mathcal{\lambda_H}(\sqrt{\varepsilon}))$ along which the eigenvalues become singular as $\varepsilon\rightarrow 0$. A particular kind of slow-fast solution known as canards (with or without head) has been found explicitly with the help of Melnikov's distance function in the blow-up space. We have also calculated an analytical expression for the maximal canard curve $(\lambda_c(\sqrt{\varepsilon}))$. Another type of periodic solution is obtained which consists of two concatenated slow and fast flow, known as relaxation oscillation. Analytically, we have proved the existence and uniqueness of the relaxation oscillation cycle using the entry-exit function \cite{Rinaldi92,Wang19AML} and validated our results numerically.

\noindent The difference in the time scale for the growth and decay in prey and predator species capture some interesting feature of respective populations. As the prey population growth takes place over a faster time scale, the predator population remains unchanged during the rapid growth and the decay of prey population. On the other hand, the change in predator population occurs slowly compared to the prey population. This type of growth and decay in two constituent species is observed for steady-state coexistence as well as oscillatory coexistence. The presence of weak Allee effect in prey growth acts as a system saver. The size of the limiting relaxation oscillation cycle is smaller in size when the magnitude of Allee effect is comparatively large (cf. Fig.~\ref{fig:canard_different_allee_strength}). It reduces the chance of extinction (maybe localized) as the periodic attractor remains away from both the axes.

To understand the change in dynamic behavior, we have chosen the predator mortality rate $\delta$, as the bifurcation parameter. For predator mortality rate greater than Hopf threshold the system stabilizes at coexistence steady state, whereas, the system shows oscillatory dynamics for mortality rate less than the threshold. For the model under consideration, the Hopf threshold is independent of the time scale parameter. But the combination of the mortality rate along with time scale parameter ($0<\varepsilon\ll1$) has enormous effect on the nature of the oscillatory coexistence. For a fixed $\varepsilon>0$, as we decrease $\delta$ below the Hopf threshold, we observe a fast transition from small amplitude oscillatory coexistence to relaxation oscillation within an exponentially small range of the parameter $\delta$ via a family of canard cycles (Fig.~\ref{fig:canard_cycles}), known as canard explosion. This type of dynamics has been observed in an ecosystem where the growth rate of the interacting species (resource-consumer type) differ on some orders of magnitude. The reason behind this can be interpreted as: when the predator density is at a maximum level, due to over-consumption the prey population collapses rapidly. Since the predator is specialist, due to the lack of food source the predator population starts decaying slowly and reaches a lower density. It reduces grazing pressure on the prey, as a result the prey population revives leading to a sudden outbreak. Again with increasing food resources, the predator population increases slowly until it reaches a desirable level which can be supported by the abundance of prey and thus the cycle continues. Empirical evidence suggests that this type of oscillatory dynamics is observed in real world, for example in the food web of Canadian boreal forest \cite{Stenseth97} where outbreak of hare population follows a cycle of almost 11 years. In the aquatic ecosystem the seasonal cycle of Daphnia and algae \cite{Scheffer00,Scheffer97} are observed and also in a forest ecosystem where insect pests defoliate the adult trees \cite{Ludwig78}.

We then considered the effect of multiple time scales in one-dimensional and two-dimensional spatial extension of our slow-fast system. 
In the 1D case, the minimum speed of the traveling wave of predator invading into the space already occupied by its prey (observed in case of compact initial conditions) is found analytically, while the patterns emerging in the wake of the front are investigated by means of numerical simulations. In the 2D case, the effect of the interplay between the weak Allee effect and the multiple timescales is studied in simulations.
The following result is worth of highlighting: 

\begin{itemize}
	
	\item in the presence of a weak Allee effect, a decrease in the time scale ratio (i.e.~for $\epsilon\ll 1$) may lead to a regime shift where the pattern becomes correlated across the whole spatial domain resulting in large-amplitude oscillations of spatially average population density; see Fig.~\ref{fig:IC2_delta0.3_eps0dot1}. Since the corresponding trajectory in the phase plane ($<u>,<v>$) comes close to the vertical axis, the immediate ecological implication of this is a likely extinction of prey.
	
\end{itemize}

On a more technical side, 
our main interest was to study how the invasion of the species is taking place and how it is getting affected with the introduction of the time scale parameter. For the values of $\delta$ less than Hopf bifurcation threshold ($\delta<\delta_H$), we find spatio-temporal chaotic patterns. The onset of spatio-temporal chaos and the duration of transient oscillation is completely influenced by the initial distribution of the two species. Fig.~\ref{fig:weak_allee_delta_0dot3} shows periodic traveling waves as transient dynamics before spatio-temporal chaos sets in. However, small amplitude heterogeneous perturbation around the homogeneous steady states reduces the time length for transient dynamics and the system quickly enters spatio-temporal chaotic regime. For $\delta$ significantly less than $\delta_H$ (see Fig.~\ref{fig:weak_allee_delta_0dot1}) we find only periodic traveling waves which indicate that continuous alteration of population patches mimics the temporal dynamics of large amplitude oscillations.

Consideration of time scale difference in the growth rates of prey and predator have some stabilizing effect on the spatio-temporal pattern formation scenario. On one hand, it increases the size of the coexisting population patches over the domain, and on other hand, it drives the spatio-temporal chaotic pattern to periodic or quasi-periodic oscillatory dynamics as shown in Fig.~\ref{fig:IC2_delta0.3_eps0dot1}. One prominent feature can be visualized from the numerical simulation that spatio-temporal chaotic pattern engulf the entire domain of size $900\times900$ at $t=1700,$ for $\varepsilon=1,$ whereas it takes quite a long time for $\varepsilon=0.1.$ The pattern obtained for $\varepsilon=0.1$ in Fig.~\ref{fig:IC2_delta0.3_eps0dot1}(b) is interacting spiral but not chaotic, as it does not show any sensitivity to initial condition. The time evolution of average prey-predator density changes from chaotic nature to quasi-periodic oscillation with the decrease in magnitude of $\varepsilon$ as shown in the lower panel of Fig.~\ref{fig:IC2_delta0.3_eps0dot1}.

In this work, we have considered a prey-predator model with specialist predator and the consumption of prey by the predator follows prey-dependent functional response. As a result, the periodic solution arising through Hopf instability is stable and there is no possibility of global bifurcation through which one or more species can collapse. The large amplitude oscillatory coexistence obtained from the temporal model changes to periodic traveling wave once the individuals are assumed to be distributed over their habitat heterogeneously. The difference in the time-scale for the growth of resource and consumer leads to the establishment of the species over larger patches and the speed of invasive wave reduces with the decrease of $\varepsilon$. The irregularity of population patches (spatio-temporal chaotic patterns) as a part of successful invasion is observed for parameter values close to the Hopf bifurcation threshold and $\varepsilon$ is close to or equal to 1. Decrease in the magnitude of $\varepsilon$ reduces the irregular oscillation but the duration of transient oscillations are enhanced.  The study of spatio-temporal pattern formation with a difference in time scales in the context of ecological systems is quite unexplored in literature. This kind of study can provide a better insight in establishment of the invasive species. More realistic phenomena can be captured if we consider long food chain model with multiple time scales and two species model with generalist predator with predator dependent functional response, which we will study in our future works.


\appendix
\section*{Appendix A} \label{App:GSPT}
Here we will follow geometric singular perturbation technique as given by Fenichel \cite{Fenichel79} to find the analytical expression of locally perturbed invariant manifold $C^1_{\varepsilon}$. Since $v=q(u,\varepsilon)$, from the invariance condition we have $$\dfrac{dv}{dt} = \dfrac{dq(u,\varepsilon)}{du}\dfrac{du}{dt}.$$  Using the explicit expression for $\dfrac{du}{dt}$ and $\dfrac{dv}{dt}$ from (\ref{eq:temp_weak_fast}) we get
\begin{equation} \label{eq:invariance_condition}
\begin{aligned}
\varepsilon q(u,\varepsilon)(u(1-\alpha \delta)-\delta) = u \dfrac{dq(u,\varepsilon)}{du}(\gamma (1-u)(u+\beta)(1+\alpha u)-q(u,\varepsilon)).
\end{aligned} 
\end{equation}
Substituting the asymptotic expansion of $q(u,\varepsilon)$ from  (\ref{eq:invariant-manifold} and assuming $u\ne0$, $\dot{q_0}(u)\ne0,$ we equate $\varepsilon$ free terms from both sides to obtain
\begin{equation}\label{eq:q0}
\begin{aligned}
q_0(u) &= \gamma (1-u)(u+\beta)(1+\alpha u),\\
\end{aligned}
\end{equation}
which is exactly the critical manifold. Now equating the coefficients of $\varepsilon$ from both sides of (\ref{eq:invariance_condition}) we get
\begin{equation}\label{eq:q1}
\begin{aligned}
q_1(u) = \dfrac{q_0(u)(u(1-\alpha \delta)-\delta)}{-u\dot{q}_0(u)}.
\end{aligned}
\end{equation}
Similarly we obtain $q_2(u)$ by equating the coefficients of $\varepsilon^2,$ 
\begin{equation}\label{eq:q2}
\begin{aligned}
q_2(u) = \dfrac{q_1(u)(u(1-\alpha \delta)-\delta)+u q_1\dot{q_1}(u)}{-u\dot{q_0}(u)}.
\end{aligned}
\end{equation}
Proceeding as above we find $q_r(u),\ r=3,4,\cdots$  by equating the coefficients of $\varepsilon^r$ from (\ref{eq:invariance_condition}).
Therefore the second order approximation of the perturbed invariant manifold is given by
$$q(u,\varepsilon)=q_0(u)+\varepsilon q_1(u)+\varepsilon^2 q_2(u),$$ where $q_0,q_1,q_2$ are given in (\ref{eq:q0})-(\ref{eq:q2}).

\section*{Appendix B} \label{App:singular_Hopf}
We apply the blow-up transformation in the slow-fast normal form (\ref{eq:slow-fast_normal_form}) where $$h_1(U,V)=u_*+U, \ \ h_3(U,V)=0,\ \ h_5(U,V)=(v_*+V)(1+\alpha u_*)+Uv_*\alpha,$$ $$h_2(U,V)=-\gamma(-1+6u_*^2\alpha+3u_*(1+\alpha(\beta-1))+\beta-\alpha \beta)-U\gamma(1+\alpha(4u_*+\beta-1)),$$ 
$$ h_4(U,V)=(v_*+V)(1-\alpha\delta_*),\ \
h_6(U,V)=u_*-(1+u_*\alpha)\delta_*,,$$
On chart $K_2$, $\bar{\varepsilon}=1$ so the blow-up transformation as defined in (\ref{eq:blow-up-map}) reduces to:
\begin{equation}\label{eq:blow-up_K2}
\bar{r} = \sqrt{\varepsilon},\ U=\sqrt{\varepsilon}\bar{U},\ V=\varepsilon \bar{V},\ \lambda=\sqrt{\varepsilon}\bar{\lambda}.
\end{equation}
Using the transformation (\ref{eq:blow-up_K2}) we can write the system (\ref{eq:desingularized_system_K2}) by removing the overbars as
\begin{equation}\label{eq:desingularized}
\begin{aligned}
U_t &= -b_1V+b_2U^2+\sqrt{\varepsilon}\mathcal{G}_1(U,V) + O(\sqrt{\varepsilon}(\lambda+\sqrt{\varepsilon})),\\
V_t &= b_3U-b_4\lambda+\sqrt{\varepsilon}\mathcal{G}_2(U,V) + O(\sqrt{\varepsilon}(\lambda+\sqrt{\varepsilon})),\\
\end{aligned}
\end{equation}
where 
\begin{equation} \label{eq:b_i's}
\begin{aligned}
\ & b_1 =u_*,\
\ b_2 =-\gamma(-1+6u_*^2\alpha+3u_*(1+\alpha(\beta-1))+\beta-\alpha \beta),\\
\ & b_3 = v_*(1-\alpha \delta_*),\
\ b_4 = v_*(1+\alpha u_*),
\end{aligned}
\end{equation}
and 
\begin{equation} \label{eq:G1 and G2}
\begin{aligned}
\mathcal{G}_1(U,V) = a_1U-a_2UV+a_3U^3,\ \ \mathcal{G}_2(u,V) = a_4U^2+a_5V.
\end{aligned}
\end{equation}	
Let the equilibrium point of the system (\ref{eq:desingularized}) is $(U_e,V_e)$, $U_e=\dfrac{b_4\lambda}{b_3}+O(2)$ and $V_e=O(2)$ where $O(2):=O(\lambda^2,\lambda\sqrt{\varepsilon},\lambda)$. Linearizing the system about this equilibrium point we have the Jacobian matrix as
\begin{equation}
\mathcal{J}:=
\begin{pmatrix}
2U_eb_2+a_1 \sqrt{\varepsilon}+O(2)&-b_1+O(2)\\
b_3+O(2)&a_5\sqrt{\varepsilon}+O(2)
\end{pmatrix}
\end{equation}
At the Hopf bifurcation we have Trace $\mathcal{J}=0$  which implies
\begin{equation}
\dfrac{2b_2b_4\lambda}{b_3}+\sqrt{\varepsilon}(a_1+a_5)+O(2)=0.
\end{equation}
and applying the blow-down map $\mathcal{\lambda_H}=\lambda\sqrt{\varepsilon}$ we get the singular Hopf bifurcation curve $\mathcal{\lambda_H}(\sqrt{\varepsilon})$ for the slow-fast normal form (\ref{eq:slow-fast_normal_form}) as 
\begin{equation}
\mathcal{\lambda_H}(\sqrt{\varepsilon})=-\dfrac{b_3(a_1+a_5)}{2b_2b_4}\varepsilon+O(\varepsilon^{3/2}).
\end{equation}

\section*{Appendix C} \label{App:maximal_canard}
Here we prove the existence of maximal canard curve and will give an analytical expression for the same. For that we will first prove the following proposition. In chart $K_2$ of the blow up space we consider the desingularized system (\ref{eq:desingularized}) as
\begin{equation}\label{eq:desingularized_r}
\begin{aligned}
U_t &= -b_1V+b_2U^2+r\mathcal{G}_1(U,V) + O(\lambda r,r),\\
V_t &= b_3U-b_4\lambda+r\mathcal{G}_2(U,V) + O(\lambda r,r),\\
r_t &=0,\\
\lambda_t &=0,
\end{aligned}
\end{equation}
where $b_1,\ b_2,\ b_3,\ b_4, \mathcal{G}_1$ and $\mathcal{G}_2$ are computed above (\ref{eq:b_i's}), (\ref{eq:G1 and G2}). The dynamics of the system on the sphere is obtained by putting $r=0$ in (\ref{eq:desingularized_r}) for different values of $\lambda$ in the vicinity of $0$. Thus, by taking $r=0,\ \lambda=0$, the above system is integrable and we have
\begin{equation} \label{eq:riccati eqn}
\begin{aligned}
U_t &= -b_1V+b_2U^2,\\
V_t &= b_3U.
\end{aligned}
\end{equation}
This is a Riccati equation and the solution of this equation helps in proving our main theorem. \\
\textbf{Proposition 1} The solution of the system (\ref{eq:riccati eqn}) is given by $H(U,V) = c,$ where $$H(U,V) = e^{-\dfrac{2b_2}{b_3}V}\Big(\dfrac{b_3}{2}U^2-\dfrac{b_1b_3^2}{4b_2^2}-\dfrac{b_1b_3}{2b_2}V\Big)$$ and
$$\dfrac{dU}{dt} = -e^{\dfrac{2b_2}{b_3}V}\dfrac{\partial H}{\partial V},$$
\begin{equation} \label{eq:H(u,v)}
\begin{aligned}
\dfrac{dV}{dt} = e^{\dfrac{2b_2}{b_3}V}\dfrac{\partial H}{\partial U}.
\end{aligned}
\end{equation} 
\begin{proof}
	We can write the above Riccati system  (\ref{eq:riccati eqn}) as 
	\begin{equation}
	\dfrac{dV}{dU}=\dfrac{b_3U}{-b_1V+b_2U^2}
	\end{equation}
	where the integrating factor is $e^{-\dfrac{2b_2}{b_3}V}$. Multiplying both sides with the I.F and integrating we get
	$$e^{-\dfrac{2b_2}{b_3}V} \Big(U^2-\dfrac{b_1}{b_2}V-\dfrac{b_1b_3}{2b_2^2}\Big)=c_0.$$
	Multiplying with $\dfrac{b_3}{2}$ we obtain the solution of the system( \ref{eq:riccati eqn}) as
	$$e^{-\dfrac{2b_2}{b_3}V}\Big(\dfrac{b_3}{2}U^2-\dfrac{b_1b_3^2}{4b_2^2}-\dfrac{b_1b_3}{2b_2}V\Big) = c,$$ where $c=c_0\dfrac{b_3}{2}$ is a constant. The solution determined by $c=0$ is a parabola of the form $$U^2=\dfrac{b_1b_3}{2b_2^2}+\dfrac{b_1}{b_2}V.$$ 
\end{proof}
\textit{Proof of theorem 4.2:}
We write the solution of the system (\ref{eq:riccati eqn}) in the parametric form 
\begin{equation}\label{eq:parametric_sol}
\eta(t) = (U(t),V(t))= \Big(t,\dfrac{b_2}{b_1}t^2-\dfrac{b_3}{2b_2}\Big),\ t\in \mathbb{R}
\end{equation}
For $\varepsilon=0$ the attracting and repelling submanifolds of the critical manifold $\mathcal{M}^1_0$ intersect along the equator of the blow-up space $S^3$. From Fenichel's theory, for $\varepsilon>0$ there exist invariant perturbed attracting $(\mathcal{M}_{\varepsilon}^{1,a})$ and repelling submanifold $(\mathcal{M}_{\varepsilon}^{1,r})$. Along the curve (\ref{eq:parametric_sol}), the attracting $(\mathcal{M}_{\varepsilon}^{1,a})$ and repelling $(\mathcal{M}_{\varepsilon}^{1,r})$ invariant submanifolds in the blow-up space intersect and the solution trajectory lying in that intersection is called maximal canard. We use Melnikov function to calculate the distance between these invariant manifolds \cite{Krupa01A}, \cite{Kuehn15} which is given by
\begin{equation}\label{eq:Melnikov_distance}
D_{r,\lambda} = d_r r + d_{\lambda} \lambda + O(r^2),
\end{equation}
where 
\begin{equation}\label{eq:melnikov_formula}
\begin{aligned}
d_r = \int_{-\infty}^{\infty}\nabla H(\eta(t))^T\mathcal{G}(\eta(t))dt,\\
d_{\lambda} = \int_{-\infty}^{\infty}\nabla H(\eta(t))^T\begin{pmatrix}
0\\-b_4
\end{pmatrix}dt,
\end{aligned}
\end{equation}
where $\mathcal{G},\ H$ and $b_4$ are defined in  (\ref{eq:G1 and G2}), (\ref{eq:H(u,v)}) and (\ref{eq:b_i's}) respectively. The distance between the submanifolds $\mathcal{M}_{\varepsilon}^{1,a}$ and $\mathcal{M}_{\varepsilon}^{1,r}$ is given by the eq.~(\ref{eq:Melnikov_distance}). And since the maximal canard lie in the intersection of these manifolds, so we must have $D_{r,\lambda}=0$.
For that we now calculate the Melnikov-type integrals $d_r$ and $d_{\lambda}$  (\ref{eq:parametric_sol}
) and (\ref{eq:melnikov_formula}). 
Therefore,
\begin{equation}
\begin{aligned}
d_r &= \int_{-\infty}^{\infty}\Big[(a_1U-a_2UV+a_3U^3 )\dfrac{\partial H(\eta(t))}{\partial U}+(a_4U^2+a_5V)\dfrac{\partial H(\eta(t))}{\partial V}\Big]dt\\
&= \int_{-\infty}^{\infty}e^{-\dfrac{2b_2}{b_3}V}\Big[(a_1U-a_2UV+a_3U^3 )b_3U+(a_4U^2+a_5V)(b_1V-b_2U^2)\Big]dt\\
&= e\int_{-\infty}^{\infty}e^{-A_4t^2}\Big(A_1t^4+A_2t^2+A_3\Big)dt
\end{aligned}
\end{equation}
where,
$$A_1 = a_3b_3-\dfrac{a_2b_2b_3}{b_1},\ A_2 = a_1b_3+ \dfrac{a_2b_3^2}{2b_2}-\dfrac{a_4b_1b_3}{2b_2}-\dfrac{a_5b_3}{2},\ A_3 = \dfrac{a_5b_1b_3^2}{4b_2^2},\ A_4 = \dfrac{2b_2^2}{b_1b_3}.$$
Now substituting $z=t^2$ and by repeated integration by parts we obtain
\begin{equation}
\begin{aligned}
d_r = e\Big(\dfrac{3A_1}{4A_4^2}+\dfrac{A_2}{2A_4}+A_3\Big)\int_{-\infty}^{\infty}e^{-A_4t^2}dt,
\end{aligned}
\end{equation}
and 
\begin{equation}
\begin{aligned}
d_{\lambda} &= -\int_{-\infty}^{\infty}b_4\dfrac{\partial H}{\partial V}dt\\
&= b_4\int_{-\infty}^{\infty}e^{-\dfrac{2b_2}{b_3}V}(-b_1V+b_2U^2)dt\\
&= e A_5\int_{-\infty}^{\infty}e^{-A_4t^2}dt,
\end{aligned}
\end{equation}
where $A_5 = \dfrac{b_1b_3b_4}{2b_2}$. Since $d_{\lambda}\ne0$ therefore using implicit function theorem we can explicitly solve for $\lambda$ from (\ref{eq:Melnikov_distance}) 
\begin{equation}
\begin{aligned}
\lambda(r) &=-\dfrac{d_r}{d_\lambda}r + O(r^2)
= -\dfrac{1}{A_5}\Big(\dfrac{3A_1}{4A_4^2}+\dfrac{A_2}{2A_4}+A_3\Big)r + O(r^2).
\end{aligned}
\end{equation}
Now using blow down map $\lambda_c=\lambda\sqrt{\varepsilon}$ we obtain the maximal canard curve for the slow-fast normal form (\ref{eq:slow-fast_normal_form}).
\begin{equation}
\begin{aligned}
\lambda_c(\sqrt{\varepsilon}) &= -\dfrac{1}{A_5}\Big(\dfrac{3A_1}{4A_4^2}+\dfrac{A_2}{2A_4}+A_3\Big)\varepsilon + O(\varepsilon^{3/2}).
\end{aligned}
\end{equation}

\section*{Appendix D} \label{App:Entry-exit}
Here we prove the existence of a unique attracting limit cycle called relaxation oscillation. To study the dynamics of the system (\ref{eq:entry_exit}) we define two section of the flow as
\begin{equation*}
\begin{aligned}
& \Delta^{in}=\{(u_+,v):u_+<<u_{max}, v\in (v_1-\rho,v_1+\rho)\},\\
& \Delta^{out}=\{(u_+,v):u_+<<u_{max}, v\in (v_0-\rho^2,v_0+\rho^2)\},
\end{aligned}
\end{equation*}
where $u_{max},\ v_1,\ v_0$ are defined in subsection \ref{subsec:entry_exit} and $\rho$ is sufficiently small positive number.

\noindent Let us define a return map $\Pi:\Delta^{in}\rightarrow \Delta^{in}$ which is a composition of two maps $$\Phi:\Delta^{in}\rightarrow\Delta^{out},\ \  \Psi:\Delta^{out}\rightarrow \Delta^{in},$$ such that $\Pi = \Psi \circ \Phi$. Let us fix $\varepsilon>0$ and we take a point $(u_+,v_+)$ on the section $\Delta^{in}$. Now we consider a trajectory of the system (\ref{eq:entry_exit}) starting from the initial point $(u_+,v_+)$. From the analysis of the entry-exit function we can say that this trajectory will be attracted to $V_+$ and will leave $V_-$ at point $(0,p(v_+)),$ where $p$ is the entry-exit function. The trajectory then jumps into the section $\Delta^{out}$ at the point $(u_+,p(v_+)).$ Thus, the map $\Phi$ is defined with the help of entry-exit function as $\Phi(u_+,v_+)=(u_+,p(v_+)).$\\
Now to study the map $\Psi$ we consider two trajectories $\gamma_{\varepsilon}^1, \gamma_{\varepsilon}^2$ starting from the section $\Delta^{out}$. These trajectories get attracted toward $C_{\varepsilon}^{1,a}$ where the slow flow is given by $\dfrac{du}{d\tau}=\dfrac{g(u,q(u,\varepsilon))}{\dot{q}(u,\varepsilon)}.$

\noindent They follow the slow perturbed manifold until the vicinity of the fold point where they contract exponentially toward each other \cite{Wang19AML} and jump into $\Delta^{in}.$ From Theorem 2.1 of \cite{Krupa01A} we have that the map $\Pi$ is a contraction. Using contraction mapping theorem we conclude that $\Pi$ has a unique fixed point which gives rise to a unique relaxation oscillation cycle $\gamma_{\varepsilon}$. Further from Fenichel's theory we infer that $\gamma_{\varepsilon}$ converges to $\gamma_0$ as $\varepsilon \rightarrow 0.$

\noindent Now for the parameter values $\alpha=0.5,\ \beta=0.2,\ \delta=0.3,$ the unique attracting cycle $\gamma_{\varepsilon}$ for $\varepsilon=0.1$, is shown below which converges to $\gamma_0$ as $\varepsilon \rightarrow 0.$
\begin{figure}[!ht]
	\begin{center}
		\mbox{\includegraphics[width=0.6\textwidth]{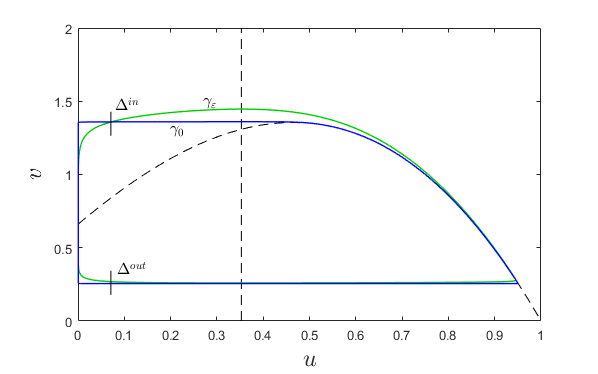}}\caption{Singular trajectory $\gamma_0$ (blue) and unique attracting limit cycle $\gamma_{\varepsilon}$ for $\varepsilon =0.1$ (green) for $\alpha=0.5,\ \beta=0.2,\ \delta=0.3.$} \label{fig:entry-exit}
	\end{center}
\end{figure}

\end{document}